\documentclass[10pt,twoside,a4paper]{amsart}
\usepackage[utf8]{inputenc}
\usepackage[inner=2.5cm,outer=3.5cm,top=2.5cm,bottom=4cm,includeheadfoot]{geometry}

\usepackage{amsmath}
\usepackage{amssymb}
\usepackage{amsopn}
\usepackage{amsthm}
\usepackage{amsfonts}
\usepackage[right]{eurosym}
\usepackage{mathtools}

\usepackage{color}
\usepackage{graphicx}

\usepackage{tikz}
\usepackage[all]{xy}

\usepackage{longtable}
\usepackage{totpages}

\usepackage[utf8]{inputenc} 
\usepackage[T1]{fontenc}

\usepackage{mathrsfs}
\usepackage{paralist}

\usepackage{tikz}
\usetikzlibrary{matrix,arrows,calc,snakes,patterns,decorations.markings}
\usepackage{tikz-cd}

\usepackage{eucal}
\usepackage{wasysym}

\usepackage{xspace} % Macht den Befehl \xspace verfügbar, der einen Blank einfügt, falls er nicht vor
\DeclareMathOperator{\rD}{D}

\DeclareMathOperator{\rp}{p}
\newcommand{\bC}{{\mathbb C}}

\newcommand{\bN}{{\mathbb N}}

\newcommand{\bP}{{\mathbb P}}
\newcommand{\bQ}{{\mathbb Q}}

\newcommand{\bZ}{{\mathbb Z}}

\newcommand{\cF}{{\mathscr F}}

\newcommand{\cL}{{\mathscr L}}

\newcommand{\cQ}{{\mathscr Q}}

\newcommand{\cW}{{\mathscr W}}
\newcommand{\cX}{{\mathscr X}}
\newcommand{\cY}{{\mathscr Y}}

\newcommand{\dO}{{\mathcal O}}
\newcommand{\dP}{{\mathcal P}}

\newcommand{\fM}{{\mathfrak M}}

\newcommand{\eps}{\varepsilon}
\renewcommand{\phi}{\varphi}

\DeclareMathOperator{\CH}{CH}
\DeclareMathOperator{\DCH}{DCH}
\DeclareMathOperator{\Xh}{\hat{X}}
\DeclareMathOperator{\iso}{\cong}

\DeclareMathOperator{\dual}{^{\vee}}					%dual-Zeichen
\DeclareMathOperator{\pr}{pr}
\DeclareMathOperator{\inj}{\hookrightarrow}
\DeclareMathOperator{\surj}{\twoheadrightarrow}

\DeclareMathOperator{\PProj}{\bP roj}
\DeclareMathOperator{\Sym}{Sym}
\DeclareMathOperator{\too}{\longrightarrow}
\DeclareMathOperator{\rank}{rk}
\DeclareMathOperator{\ch}{ch}
\DeclareMathOperator{\td}{td}
\DeclareMathOperator{\Pic}{Pic} 
 
\DeclareMathOperator{\id}{id}

\DeclareMathOperator{\del}{\partial}

\DeclareMathOperator{\Def}{Def}

\DeclareMathOperator{\Hilb}{Hilb}
\DeclareMathOperator{\an}{an}

\DeclareMathOperator{\Gammabar}{\overline{\Gamma}}

\newcommand\restr[2]{{% we make the whole thing an ordinary symbol
  \left.\kern-\nulldelimiterspace % automatically resize the bar with \right
  #1 % the function
  \vphantom{\big|} % pretend it's a little taller at normal size
  \right|_{#2} % this is the delimiter
  }}

\newif\ifmyversion
 \myversionfalse
\newenvironment{note}{\color{gray} \ \\(Anmerkung nur für mich: }{) } %benuten mit: \ifmyversion ... \else \fi 

\newcommand{\TODO}[1]{}%{\textcolor{red}{(Todo: {\it #1})}\ }

\newcommand{\Martin}[1]{}%{\textcolor{green}{(@Martin: {\it #1})}\ }

\newdir^{ (}{{}*!/-5pt/@^{(}}  %Für injektiv-Pfeile: kann man jetzt tippen @{^{ (}->} und dann ist das spacing schöner.
\newdir_{ (}{{}*!/-5pt/@_{(}}

\newcommand{\I}{{\rm{(I)}}}
\newcommand{\II}{{\rm (II)}}
\newcommand{\III}{{\rm{(III)}}}
\newcommand{\stern}{{\rm{(*)}}}
\newcommand{\sternn}{{\rm{(**)}}}
\newcommand{\sternnn}{{\mathrm{(*\!*\!*)}}}

\theoremstyle{plain}
\newtheorem{proposition}{Proposition}[section]

\newtheorem{lemma}[proposition]{Lemma}
\newtheorem{claim}[proposition]{Claim}

\newtheorem{corollary}[proposition]{Corollary}

\newtheorem{conjecture}[proposition]{Conjecture}

\newtheorem{theorem}[proposition]{Theorem}

\theoremstyle{definition}
\newtheorem{definition}[proposition]{Definition}

\newtheorem{remark}[proposition]{Remark}

\theoremstyle{remark}

\newtheoremstyle{name}
   {}{}{\itshape}{}{\bfseries }{}{ }{\thmname{#3}.}
\theoremstyle{name}
\newtheorem{name}{bla}

\numberwithin{equation}{section}					% Nummerierung der Equations

\begin{document}
\title[On the Chow ring of birational irreducible symplectic varieties]{On the Chow ring of birational irreducible symplectic
  varieties}
\author[U. Rie{\ss}]{Ulrike Rie\ss}
\address{Mathematical Institute, Endenicher Allee 60, 53115 Bonn, Germany}
%\email{uriess@math.uni-bonn.de}
\maketitle
\let\thefootnote\relax\footnotetext{Funded by the SFB/TR 45 `Periods,
moduli spaces and arithmetic of algebraic varieties' of the DFG
(German Research Foundation)}

\begin{abstract} 
  We show that the graded Chow rings of two birational irreducible symplectic
  varieties are isomorphic. This lifts a result known for the cohomology algebras to the level of Chow
  rings, despite the non-injectivity the cycle class map. In the special case of general Mukai
  flops, we present an alternative approach based on explicit calculations.
\end{abstract}

%%%%%%%%%%%%%%%%%%%%%%%%%%%%%%%%%%%%%%%%%%%%%%%%%%%%%%%5 Introduction %%%%%%%%%%%%%%%%%%%%%%%%%%%%%%
%%%%%%%%%%%%%%%%%%%%%%%%%%%%%%%%%%%%%%%%%%%%%%%%%%%%%%%%%%%%%%%%%%%%%%%%%%%%%%%%%%%%%%%%%%%%%%%%%%%
\section{Introduction}
An {\it irreducible symplectic variety} or {\it algebraic hyperkähler manifold} is a simply
connected, nonsingular,  complex projective variety with a nowhere degenerate two-form $\sigma$
generating $H^0(X,\Omega_X^2)$.
Two important series of examples are provided by Hilbert schemes of points Hilb$^n(S)$ on a K3 surface $S$
and generalized Kummer varieties $K_n(A)$ associated to an abelian surface $A$.

The main result of this article is the following (see Theorem \ref{thm:invertibility}):
\begin{name}[Theorem]
  Let $X$ and $X'$ be birational irreducible symplectic varieties. Then there exists a correspondence $[Z]_*:\CH(X) \overset{\iso}{\too}
  \CH(X')$ which is an isomorphism of graded rings.
\end{name}

The most important aspect of this theorem %Theorem \ref{thm:invertibility} 
is the multiplicativity of $[Z]_*$. The fact
that the Chow groups of $X$ and $X'$ are isomorphic as additive groups (without regard to the multiplicative structure) was already known as a
consequence of \cite[Theorem 3.2]{FuWang08}.

Consider deformations $\cX$ and $\cX'$ of $X$ and $X'$ which are
isomorphic away from the special fibre
(see \cite[Theorem 4.6]{Huybrechts:HK:basic-results})
 and let $Z$ be the limit of the graphs of isomorphisms $\cX_t
\iso \cX'_t$.  Then $[Z]$ is known to yield an isomorphism $[Z]_*^H:H^*(X,\bZ)
\overset{\cong}{\too}H^*(X',\bZ)$ of graded rings (cf. Section \ref{ssec:cohomology}).
Using the cycle class map, one could try to deduce Theorem \ref{thm:invertibility} from this statement.
However, the kernel of the cycle class map  $\CH(X)\to H^*(X,\bZ)$ is very big. Already for a K3 surface
it is infinite dimensional, due to the classical result of Mumford \cite{Mumford69}.

Instead, we use specialization for Chow rings in families to prove multiplicativity and invertibility of
$[Z]_*$ directly (cf. Section \ref{ssec:proofofthm}).

As an application of Theorem \ref{thm:invertibility}, we study questions related to the Bloch--Beilinson
conjecture, which was in fact our original motivation.  We first observe that for irreducible symplectic
varieties the termination of the conjectural Bloch--Beilinson filtration is invariant under birational
correspondences.  

Furthermore we study conjectures of Beauville and Voisin on the Chow rings of irreducible
symplectic varieties. In \cite{Beauville} Beauville considers the subalgebra $\DCH(X)\subseteq
\CH_\bQ(X)$ generated by $\CH^1_\bQ(X)$, and predicts that the restriction of the cycle class map $\restr{c_X}{\DCH(X)}:\DCH(X)\inj
H^2(X,\bQ)$ is injective. Voisin
extends this conjecture to the bigger subalgebra which also includes the Chern classes $c_i(T_X)$ (see
\cite{Voisin-08}).
% In \cite[Proposition 2.6]{Beauville}
%Beauville proves the invariance of his conjecture under elementary Mukai flops. 
Using the above theorem, we show (see Theorem \ref{thm:generalize Beauville}):
\begin{name}[Theorem] 
  The conjectures of Beauville and Voisin are both invariant under birational
  correspondences.
\end{name}

The most fundamental examples of birational correspondences between irreducible symplectic varieties
$X$ and $X'$ are provided by general Mukai flops. In this case one can fix families $\cX$ and $\cX'$ as
above and explicitly determine $Z$. Thus computing the action of $[Z]_*$ provides an
alternative approach to Theorem \ref{thm:invertibility}. In order to demonstrate that even in this fundamental case the result is
non-trivial, we will show the multiplicativity of $[Z]_*$ in this case by explicit computations. This
will take up all of Sections \ref{sec:general facts} and \ref{sec:multiplicativity}.

Still in the case of general Mukai flops, the last section relates Theorem \ref{thm:invertibility} to a
known result on derived categories and Grothendieck groups.  Namikawa proved that the Fourier--Mukai
transform $\Phi_{\dO_Z}:\rD^b(X) \to \rD^b(X')$ is an equivalence of categories (cf. \cite[Theorem
5.1]{Namikawa2003}). We show that the induced map $\Phi^{\CH}_{v(\dO_Z)}:\CH_\bQ(X)\to \CH_\bQ(X')$
coincides with $[Z]_*$ % (see Section \ref{subsec:link to D})
and thus deduce that
$\Phi^K_{[\dO_Z]}:K(X)\otimes \bQ\to K(X')\otimes \bQ$ is an isomorphism of graded rings. This
multiplicativity is not reflected on the level of derived categories, since $\Phi_{\dO_Z}$ is not
compatible with the derived tensor product (cf. \cite{Balmer}).

Note that in \cite[Theorem 4.6]{Huybrechts:HK:basic-results} Huybrechts showed the existence of 
deforming families $\cX$ and $\cX'$ as non-projective complex manifolds. 
In order to use classical intersection theory, we show in Section \ref{sec:preparations} that such families indeed exist as algebraic spaces.

\bigskip
\noindent {\bf Acknowledgements.}
I wish to thank my advisor Daniel Huybrechts for his support. Moreover, I thank Baohua Fu and
Roland Abuaf for their comments on the first version of this article. Finally, I would like to thank the
referee for his suggestions. Sections \ref{sec:general facts}
and \ref{sec:multiplicativity} are part of the author's diploma thesis.

%%%%%%%%%%%%%%%%%%%%%%%%%%%%%%%%%%%% Preparations  %%%%%%%%%%%%%%%%%%%%%%%%%%%%%%%%
%%%%%%%%%%%%%%%%%%%%%%%%%%%%%%%%%%%%%%%%%%%%%%%%%%%%%%%%%%%%%%%%%%%%%%%%%%%%%%%%%%%%%%%%%%%%%%%%%%%
\section{Preparations} \label{sec:preparations}
Throughout the article algebraic spaces will be separated algebraic spaces of finite type over
$\bC$. For the definition and properties of algebraic spaces, we refer to \cite{Knutson}. 
 We will denote by $\CH(X)$ the Chow ring (with integral coefficients) of a
nonsingular integral algebraic space $X$, whereas the Chow ring with
coefficients in $\bQ$ will be denoted by $\CH_\bQ(X)$.  
By the term ``complex variety'' we refer to a separated
integral scheme of finite type over $\bC$.

In this section we will lay the foundations for the proof of the main theorem. We show that deforming
families as in \cite[Theorem 4.6]{Huybrechts:HK:basic-results} exist in the category of algebraic spaces, and we
briefly recall intersection theory for algebraic spaces, including specialization maps.

%%%%%%%%%%%%%%%%%%%%%%%%%%%%%%%%%%%%%%%%%%%%%%%%%%%%%%%%%%%%%%%%%%%%%%%%%%%%%%%%%%%%%%%%%%%%%%%%%%%
\subsection{Existence of $\cX$ and $\cX'$ as algebraic spaces}\label{ssec:algexistence}

The following proposition is an algebraic version of \cite[Theorem 4.6]{Huybrechts:HK:basic-results}. 
 It is essential for many of the proofs.
\begin{proposition}\label{prop:families}
  Let $X$ and $X'$ be birational irreducible symplectic varieties. Then there exist families of smooth
  integral algebraic spaces $\cX$ and $\cX'$ over $T$, smooth quasi-projective one-dimensional
  complex variety, and a closed point $0\in T$ such that
  \begin{compactenum}
  \item $\cX_0=X$ and $\cX'_0=X'$, and
  \item there is an isomorphism $\Psi:\cX_{T\setminus \{0\}} \iso \cX'_{T\setminus \{0\}} $ over $T$.
  \end{compactenum}
\end{proposition}

For the proof we need to work with not necessarily projective hyperkähler manifolds. 
By {\it hyperkähler manifold} we refer to a simply
connected, compact Kähler manifold $X$, such that  $H^0(X,\Omega_X^2)$ is generated by a nowhere
degenerate two-form. 
The definition of an irreducible symplectic variety is recovered from this by additionally requiring projectivity.

The second integral cohomology $H^2 (X,\bZ)$ of a hyperkähler manifold $X$ is endowed with an integral
quadratic form, called the Beauville--Bogomolov form, which we denote by $q$. A detailed overview on
hyperkähler manifolds can be found in \cite[Part III]{Gross-Huybrechts-Joyce} and \cite{Huybrechts:HK:basic-results}.

\begin{proof}[Proof of Proposition \ref{prop:families}]
In order to increase the readability of the proof, the following diagram contains most of the maps that
will occur: 
 $$
  \begin{tikzcd}[column sep=small, row sep=tiny]
       &\cX_{\Hilb} \ar{dd} \ar[hookleftarrow]{rr} &&\cX_{H^{\an}} \ar{dd}\ar[two heads]{rr} 
       &&\cX_{D_L}\ar{dd} \ar[hook]{rr} &&\cX_{D^{\an}} \ar{dd}
       &&\cX'_{D'^{\an}}\ar{dd} \ar[hookleftarrow]{rr} &&\cX'_{D_{L'}}\ar{dd}\\
       \cX \ar{dd} \ar{ur}\ar[hookleftarrow, crossing over]{rr}&& \cX_{\an} \ar{ur} &&&&&&&&&\\
       &\Hilb^X_{\bP^n} \ar[hookleftarrow]{rr}&& H^{\an}  \ar[two heads]{rr}{\eta} 
       && D_L  \ar{ddrrr}[swap]{\dP}\ar[hook]{rr} &&D^{\an}\ar{ddr}{\dP}
       && D'^{\an} \ar{ddl}[swap]{\dP'}\ar[hookleftarrow]{rr}&&D_{L'} \ar{ddlll}{\dP'} \ .\\
       T\ar{ur}{f}\ar[hookleftarrow]{rr}&&T^{\an}\ar{ur}[swap]{f} \ar[leftarrow,crossing over]{uu}&&&&&&&&&\\
       &&&&&&&&\cQ&&&
  \end{tikzcd}
  $$

  Let $X$ and $X'$ be birational irreducible symplectic varieties and fix a very ample $L\in \Pic (X)$. 
  Consider the induced embedding $X\inj \bP^N$. Let
  $\Hilb^X_{\bP^n}$ be the irreducible component of the Hilbert scheme containing the
  point $[X]$, and $\cX_{\Hilb}$ be the corresponding universal family. Denote the pullback of $\dO(1)$
  by $\cL \in \Pic(\cX_{\Hilb})$.

  Since the Kähler manifold $X$ deforms unobstructed (see e.g. \cite{Kawamata92} and \cite{Ran92}), there
  exists a local
  deformation space $D^{\an}:=\Def(X)$ of $X$ together with a universal family
  $\cX_{D^{\an}}$. After shrinking, we may assume that $D^{\an}$ is contractible.
  Let $D_L:=\Def(X,L)\subseteq D^{\an}$ be the subset
  parametrizing deformations of the pair $(X,L)$.
  The
  choice of $L$ implies that (up to shrinking $D_L$) there exists an open subset $H^{\an}
  \subseteq \Hilb^X_{\bP^n}$ such that the restriction $\cX_{H^{\an}}$ of the universal family
  to $H^{\an}$ induces a proper surjective morphism $\eta:H^{\an} \surj D_L$. 
 
  For general $t\in D_L$ the Picard rank is $\rho(\cX_t)=1$. Thus, the same is true for general $t\in H^{\an}$, and
  general $t \in \Hilb^X_{\bP^n}$. Therefore, there is a smooth quasi-projective curve $T$ with a map $f:T\to
  \Hilb^X_{\bP^n}$ such that some point $0\in T$ is mapped to $[X]$, and $\rho(\cX_t)=1$ for a general
  element $t\in T$. Let $\cX$ be the pullback of $\cX_{\Hilb}$ to $T$. By shrinking $T$ we may
  assume that $\cX$ is smooth.
  Define further $T^{\an}:=f^{-1}(H^{\an})$, and $\cX_{\an}$ as the pullback of $\cX$ to
  $T^{\an}$.
 
  Fix a birational map $\phi:X\dashrightarrow X'$. This is an isomorphism away from a set of codimension
  at least two%(use uniqueness of symplectic form; see e.g. \cite[Section
              %2.2]{Huybrechts:birHK_deformations})
, and therefore induces an isomorphism $\phi^*:H^2(X',\bZ)\to H^2(X,\bZ)$
  (see e.g. \cite[Lemma 2.6]{Huybrechts:HK:basic-results}). Fix a lattice $\Lambda$ which is isomorphic
  to $H^2(X,\bZ)$, together with a marking $g$ of
  $X$ (i.e. an isomorphism $g:H^2(X,\bZ)\to \Lambda$). This induces a marking $g':= g\circ\phi^*$ of $X'$. 

  Similar as above, $X'$ deforms to $\cX'_{D'^{\an}}\to D'^{\an}:=\Def(X')$, and if $L'\in \Pic(X')$ corresponds
  to $L$ (via $\phi^*$), the pair $(X',L')$ deforms to a subset $D_{L'}:=\Def(X',L')$.

    Consider the period map 
  $\dP: D^{\an}\to\cQ = \{x | \; q(x)=0 , q(x + \bar{x})>0\}\subset \bP(\Lambda \otimes \bC)$, sending
  $t\in D^{\an}$ to $\bP(g_t(H^{2,0}(\cX_t)))$, where $g_t$ is the marking induced by parallel
  transport% (up to shrinking, $D^{\an}$ is symply connected)
  . Analogously, define
  $\dP':D'^{\an} \to \cQ$.  The Local Torelli Theorem states that these maps are local isomorphisms (see
  \cite[Theorem 5]{Beauville1983}). Therefore, (up
  to further shrinking $D^{\an}$ and $D'^{\an}$)
  they induce an isomorphism $\dP'^{-1}\circ \dP:D^{\an} \to D'^{\an}$, which identifies $D_L$ with $D_{L'}$.

  \begin{claim}\label{clm:birationality}
   For any $t\in
    \dP(D^{\an})\subseteq \cQ$, the manifolds $\cX_t:= (\cX_{D^{\an}})_{\dP^{-1}(t)}$ and
    $\cX'_{t}:=(\cX'_{D'^{\an}})_{\dP'^{-1}(t)}$ are birational.
  \end{claim}

\begin{proof}
  For $t \in \dP(D^{\an})$, let again  
  $g_t:H^2(\cX_t,\bZ)\to \Lambda$ and $g'_t:H^2(\cX'_t, \bZ)\to \Lambda$ be the
    markings induced from $g$ and $g'$ by parallel transport.
  Let $U\subseteq \dP(D^{\an})$ be the subset of elements $u\in \dP(D^{\an})$, for which there exists an
  isomorphism $\phi_u:\cX_u\iso \cX'_u$ such that $\phi_u^*=g_u^{-1} \circ g'_u$, i.e.
 for $u\in U$, the pairs $(\cX_u,g_u)$ and $(\cX'_u,g'_u)$
  correspond to the same point in the moduli space $\fM$ of marked hyperkähler manifolds. 

  In the following, we will show that $U$ is a dense open subset in
  $\dP(D^{\an})$. Therefore, for all $t\in \dP(D^{\an})$ the marked hyperkähler manifolds $(\cX_t,g_t)$ and
  $(\cX'_t,g'_t)$ correspond to non-separated points in $\fM$ and are thus birational (see
  \cite[Theorem 4.3]{Huybrechts:HK:basic-results}).

  First observe that $U$ is non-empty, since by \cite[Theorem 4.6']{Huybrechts:HK:basic-results} $(X,g)$ and $(X',g')$ are
  non-separated points in $\fM$.
  By the Local Torelli Theorem, $U$ is open. %in classical topology

  In order to see that $U\subseteq \dP(D^{\an})$ is dense, consider the set $W:=\{w\in \dP(D^{\an})\,|\,
  \rho(\cX_w)=0 \}\subseteq \dP(D^{\an})$. This is a complement of a union of countably many hypersurfaces. Since
  hypersurfaces are of real codimension two, the set $W$ is still
  connected. 
  Furthermore $W\subseteq \dP(D^{\an})$ is dense, and therefore $W\cap U \neq \emptyset$.
  The same argument as in the proof of \cite[Theorem 5.1]{Huybrechts:HK:basic-results} shows that
  $W\cap \del U= \emptyset$, where $\del U:=\overline{U}\setminus U$ denotes the border of
  $U$. Therefore $U$ contains $W$, and in particular it is dense in $\dP(D^{\an})$. This proves Claim
  \ref{clm:birationality}.
\end{proof}

  Define $\cX'_{\an}$ as the pullback of $\cX'_{D'^{\an}}$ to $T^{\an}$
  along the composition $\dP'^{-1}\circ \dP\circ \eta \circ f$. This comes with $\cL' \in
  \Pic(\cX'_{\an})$. Denote the fibre over an element $t\in T^{\an}$ by $\cX'_t:=(\cX'_{\an})_t$, and
  let $\cL'_t:=\restr{\cL'}{\cX'_t}$.
  For general $t\in T^{\an}$, both $\cL_t$ and $\cL'_t$ have
  global sections by semicontinuity, %$\cL_t$ is very ample,
  $\rho({\cX}_t)=\rho(\cX'_t)=1$ (this uses the choice of $T$), and $0<q(\cL_t)=q(L)=q(L')=q(\cL'_t)$ (see \cite[Corollary
  2.7]{Huybrechts2003}).
  
  Claim \ref{clm:birationality} implies the existence of a birational map $\phi:\cX_t\dashrightarrow
  \cX'_t$. The pullback $\phi^*(\cL'_t)$ has nontrivial global sections and satisfies $
  q(\phi^*(\cL'_t))=q(\cL'_t)=q(\cL_t)$. Therefore we can conclude that $\phi^*(\cL'_t)=\cL_t$.
  The projectivity criterion for hyperkähler manifolds \cite[Theorem 2]{Huybrechts:HK:basic-results:erratum} implies that
  $\cL_t=\phi^*(\cL'_t)$ and $\cL'_t$ are ample line bundles. Therefore, $\phi$ extends to an isomorphism
  $\phi:\cX_t\iso \cX'_t$.

  The set $V\subseteq T^{\an}$, where $\cL_v$ and $\cL'_v$ are ample and define an isomorphism 
  $\cX_v\iso \cX'_v$ is open (see \cite[Theorem 1.2.17]{LazarsfeldI}),
 and by shrinking $T^{\an}$ we may assume that $V=T^{\an}\setminus
  \{0\}$. Then
  $\restr{\cX'_{\an}}{T^{\an}\setminus \{0\}} \iso  \restr{\cX_{\an}}{T^{\an}\setminus \{0\}}$.
  Finally, define a complex manifold $\cX'$ by gluing $\cX'_{\an}$ into $\cX \setminus X$ along this isomorphism.
  Clearly, these $\cX$ and $\cX'$ satisfy the conditions of Proposition
\ref{prop:families} and it suffices to show, that $\cX'$ is an algebraic space.

Since $\cX$ is quasi-projective,
we can consider its closure $\overline{\cX}$ with respect to an arbitrary embedding into a projective
space.  Then define a complex space $\overline{\cX'}$ by gluing $\overline\cX \setminus X$ with $\cX'$
along $\cX\setminus X$. With this construction $\overline{\cX'}$ is a Moishezon space, since
$\overline{\cX'}$ is birational to $\overline{\cX}$, which is projective.
Therefore, \cite[Theorem 7.3]{Artin:blow-down} implies that $\overline{\cX'}$ corresponds to an algebraic
space. Conclude the proof of Proposition \ref{prop:families} by observing that the (Zariski-)open subset $\cX'$ consequently also exists as
algebraic space.
\end{proof}

%%%%%%%%%%%%%%%%%%%%%%%%%%%%%%%%%%%%%%%%%%%%%%%%%%%%%%%%%%%%%%%%%%%%%%%%%%%%%%%%%%%%%%%%%%%%%%%%%%%
\subsection{Intersection theory for algebraic spaces}\label{ssec:inttheo-sp}
As Edidin and Graham pointed out in \cite[Section 6.1]{Edidin-Graham}, the  whole intersection theory as
presented in \cite[Chapters 1-6]{Fulton} still works in the category of algebraic spaces. For an
algebraic space, codim-$k$-cycles are defined as formal sums with $\bZ$-coefficients of integral closed
subspaces of codimension $k$ in $X$. 
Rational equivalence  is generated by div$(\phi)$ for rational functions on codimension $k-1$ subspaces
$W\subseteq X$. Here, for $Y\subseteq W$ of codimension one and $\phi\in K(W)$, the multiplicity mult$_Y (\phi)$ can be defined by pulling  back to a representable \'etale covering
(i.e. if $f: U\to X$ is a representable \'etale covering, then mult$_Y (\phi) := {\rm mult}
_{f^{-1}(Y)}(f^* \phi)$). With these definitions the results of \cite[Chapters
1-6]{Fulton} hold in the category of algebraic spaces. In particular, Chow rings of algebraic spaces have
the known functorial properties (proper push-forward, Gysin morphism, pull-backs), and consequently 
every smooth integral algebraic space admits an intersection pairing.

Analogously, specialization still works in the category of algebraic spaces:

%%%%%%%%%%%%%%%%%%%%%%%%%%%%%%%%%%%%%%%%%%%%%%%%%%%%%%%%%%%%%%% The specialization map
\subsubsection{The specialization map} \label{sec:specialization} 
%Alternativ könnte ich auch einfach nur die Spezialisierung von $\cY\setminus \cY_t \to \cY_t$ benutzen,
%wie sie in Fulton Example 6.3.7 angeführt ist.
Let us fix the notation for the specialization maps.
Let $T$ be a smooth 
one-dimensional integral (separated) algebraic space and $\eta$ be its generic
point. 
 Let $\pi:\cY \to T$ be a smooth morphism of integral algebraic spaces. 
Fix a $\bC$-rational point $t\in T$. 
  There is a commutative
  triangle:
  $$
  \begin{tikzcd}[column sep=tiny]
       &\CH(\cY) \ar{dl}[swap]{r_\eta}\ar{dr}{s_t} & \\
    \CH(\cY_\eta)\ar{rr}{\sigma} & & \CH(\cY_t) \,,
  \end{tikzcd}
  $$
  where $s_t$ denotes the restriction to $\cY_t$ (i.e.\,the pull-back to the special fibre, which coincides with the specialization map of
  \cite[Chapter 10.1]{Fulton}). On the level of cycles, $r_\eta$ is defined as the restriction to the
  generic fibre and $\sigma$ as the composition of taking the closure in $\cY$ and restriction to
  $\cY_t$. The map $\sigma$ is called {\it specialization map}. Commutativity of the triangle and
  compatibility of $r_\eta$ and $\sigma$ with rational equivalence may be checked explicitly.

  All three maps are compatible with the intersection product, pull-back, proper push-forward, and taking Chern classes.

  By slight abuse of notation, we will not keep the family $\cY$ in the notation, but use the symbols
  $r_\eta$, $s_t$, and $\sigma$ for various families.

%%%%%%%%%%%%%%%%%%%%%%%%%%%%%%%%%%%% Main Theorem  %%%%%%%%%%%%%%%%%%%%%%%%%%%%%%%%
%%%%%%%%%%%%%%%%%%%%%%%%%%%%%%%%%%%%%%%%%%%%%%%%%%%%%%%%%%%%%%%%%%%%%%%%%%%%%%%%%%%%%%%%%%%%%%%%%%%
\section{Main theorem} \label{sec:maintheorem}
In this section we state and prove Theorem \ref{thm:invertibility}, which is the main result of this
article.

%%%%%%%%%%%%%%%%%%%%%%%%%%%%%%%%%%%%%%%%%%%%%%%%%%%%%%%%%%%%%%%%%%%%%%%%%%%%%%%%%%%%%%%%%%%%%%%%%%%
\subsection{Notation and formulation of the theorem}\label{ssec:notation+theorem}

Fix the following notation for the rest of the article:
\begin{name}[Notation]{ \rm
For birational irreducible symplectic varieties $X$ and $X'$ fix families $\cX\to T$ and $\cX'\to T$ as in
Proposition \ref{prop:families}. 
  Denote the generic point of $T$ by $\eta$. The generic fibres of $\cX$ and $\cX'$ are consequently
  denoted by $\cX_\eta$ and $\cX'_\eta$ respectively. Then $\Psi$ restricts to an isomorphism $\psi:
  \cX_\eta \to \cX_\eta'$. Let $\Gamma$ be the graph of $\psi$, and $\Gammabar\subseteq \cX \times_T\cX'$
  be its closure. Finally, define $Z\subseteq X\times X'$ as the special fibre of $\Gammabar$ and denote
  the associated class in $\CH(X\times X')$ by $[Z]$.
}
\end{name}

Let $q$ and $q'$ be the projections from $X\times X'$ to $X$ and $X'$ respectively.
\begin{definition}  \label{def:[Z]_*}
Define $[Z]_*: \CH(X) \to \CH(X')$ as the correspondence with kernel $[Z]$, i.e. as the map given by $[Z]_*(\alpha):= q'_*([Z].q^*\alpha)$ for all $\alpha \in \CH(X)$.
\end{definition}

We can now state the main result of this article, which is new even for the case of elementary Mukai
flops:
\begin{theorem}\label{thm:invertibility} 
  Let $X$ and $X'$ be birational irreducible symplectic varieties.
  Then the map $[Z]_*:\CH(X) \to \CH(X')$ is an isomorphism of graded rings. Its inverse is the
  correspondence $[Z]_*^t:\CH(X') \to \CH(X)$ with kernel $[Z]$ in the opposite direction.
\end{theorem}
The proof of this theorem is given in Section \ref{ssec:proofofthm}.

\begin{remark}
  The fact that $X$ and $X'$ are irreducible symplectic varieties is only used in the proof of Theorem
  \ref{thm:invertibility} in order
  to deduce the existence of families as in Proposition \ref{prop:families}. Therefore, the theorem holds
  more generally, whenever the existence of such families is known.
\end{remark}

\begin{remark}
  Instead of $Z$, one could consider $\overline{\Delta}\subseteq X\times X'$, where $\Delta$ is the graph
  of a birational isomorphism. However, the map $\big[\overline{\Delta}\big]_*$ is not multiplicative in
  general (deduce its non-multiplicativity in the case of elementary Mukai flops e.g. from the
  computations of \cite[Example 6.6]{LeeLinWang}).
\end{remark}

%%%%%%%%%%%%%%%%%%%%%%%%%%%%%%%%%%%% Theorem (inv) in cohomology  %%%%%%%%%%%%%%%%%%%%%%%%%%%%%%%%
\subsection{Theorem \ref{thm:invertibility} in cohomology} \label{ssec:cohomology}
Before proving Theorem \ref{thm:invertibility}, we discuss the analogous statement in cohomology.

For the purpose of this subsection it is enough to work in the more general setting of (not necessarily
projective) complex manifolds. 

Let $X$ and $X'$ be birational compact hyperkähler manifolds. Then by 
\cite[Theorem 2.5]{Huybrechts2003} there exist deforming families of complex manifolds
$\cX$ and $\cX'$ satisfying analogous conditions as in Proposition \ref{prop:families}. Define $Z$ as above and
let $[Z] \in H^*(X\times X')$ be the cohomology class of the analytic cycle $Z$. This induces a correspondence 
$[Z]^H_*:H^*(X,\bZ) \to H^*(X',\bZ)$. The analogous statement to Theorem \ref{thm:invertibility} on the
level of cohomology is:
\begin{lemma}[cf. {\cite[Corollary 2.7]{Huybrechts2003}}]
  The map $[Z]^H_*:H^*(X,\bZ) \to H^*(X',\bZ)$ is an isomorphism of graded rings.
\end{lemma}
\begin{proof}
  This follows from Ehresmann's Theorem (see e.g. \cite[Theorem 9.3]{VoisinI}), since the cycle $[Z]$ is by
  definition the limit cycle of the graphs of the isomorphisms $\cX_t\iso \cX'_t$ for $t\neq
  0$. 
\end{proof}

%%%%%%%%%%%%%%%%%%%%%%%%%%%%%%%%%%%% Proof of Theorem (invertibility)  %%%%%%%%%%%%%%%%%%%%%%%%%%%%%%%%
\subsection{Proof of Theorem \ref{thm:invertibility}} \label{ssec:proofofthm}
  In order to prove the theorem, one needs to show:
  \begin{itemize}
  \item {\it (Compatibility with graduation):} The map $[Z]_*$ respects the grading of the Chow rings,
    i.e. $[Z]_*\big(\CH^k(X)\big)\subseteq \CH^k(X')$.
  \item {\it (Invertibility):} The maps $[Z]_*$ and $[Z]_*^t$ are inverse.
  \item {\it (Multiplicativity):} The map $[Z]_*$ is multiplicative, i.e. for all $\alpha, \beta \in
    \CH(X)$ the equality $[Z]_*(\alpha).[Z]_*(\beta)=[Z]_*(\alpha.\beta)$ holds.
  \end{itemize}

The compatibility with the graduation follows from the fact, that $Z$ is of pure dimension dim$(X)$.

\vspace{0.3 em}\noindent 
{\it Proof of invertibility.}
  By symmetry of the situation, we only need to
  show that $[Z]_*^t\circ [Z]_*= \id$.
  The map $[Z]_*^t\circ [Z]_*$ is the correspondence with kernel $\alpha_0:={\pr_{13}}_*\big([Z\times
  X].[X\times Z^t]\big)\in \CH(X\times X)$, where $\pr_{13}:\CH(X\times X'\times X)\to \CH(X\times X)$ is
  the projection to the first and third factor (cf. \cite[Section 16.1]{Fulton}).
  We will show that $\alpha_0= [\Delta_X]$, using the existence of families as in Proposition
  \ref{prop:families} and the specialization map (cf. Section \ref{sec:specialization}).

  Let $\cX$ and $\cX'$ be families as in Proposition \ref{prop:families} and keep the notation of Section
  \ref{ssec:notation+theorem}. Consider the cycle 
  $$\alpha:= {\pr_{13}}_*\big([\Gammabar\times_T \cX].[\cX\times_T \Gammabar^t]\big)\in
  \CH(\cX\times_T \cX),$$ where once again $\pr_{13}$ is the projection to the first and third factor.
  Its restriction to the generic fibre is $\alpha_\eta :=r_\eta(\alpha)= 
 {\pr_{13}}_*\big([\Gamma\times_{k(\eta)} \cX_\eta].[\cX_\eta\times_{k(\eta)}
  \Gamma^t]\big)\in \CH(\cX_\eta\times_{k(\eta)} \cX_\eta)$. Using the fact that $\Gamma$ is the graph of
  an isomorphism, this can explicitly be determined as: $\alpha_\eta=
  [\Delta_{\cX_\eta}]$.

Since, moreover, the restriction of $\alpha$ to the special fibre is $s_0(\alpha)=\alpha_0$, this allows us to conclude:
$$ \alpha_0 = s_0(\alpha) = \sigma(\alpha_\eta) =\sigma([\Delta_{\cX_\eta}])=[\Delta_{X}] ,
$$
thus proving the invertibility of $[Z]_*$.

\vspace{0.3 em}\noindent 
{\it Proof of multiplicativity.}  Let $\Delta_3\subseteq X\times X\times X$ be the small diagonal,
i.e. the image of the natural inclusion $X\inj X\times X\times X$ and denote the small 
diagonal in $X'\times X'\times X'$  by $\Delta_3'$.  Consider the following diagram:
  $$
  \begin{tikzcd}[column sep = large]
    \CH(X)\times \CH(X) \ar{r}{[Z]_*\times[Z]_*}\ar{d}[swap]{\times} &\CH(X')\times \CH(X')\ar{d}{\times}\\
    \CH(X\times X)\ar{r}{[Z\times Z]_*} \ar{d}[swap]{[\Delta_3]_*} &\CH(X'\times X') \ar{d}{[\Delta_3']_*}\\
    \CH(X) \ar{r}{[Z]_*}  & \CH(X') \ .
  \end{tikzcd}
  $$

We will show the following:
\begin{enumerate}[(a)]
\item The composition $[\Delta_3]_* \circ \times$ coincides with the multiplication and the same
  holds for $[\Delta_3']_* \circ \times$, \label{it:a}
\item the upper rectangle is commutative, and \label{it:b}
\item also the lower rectangle is commutative.\label{it:c}
\end{enumerate}
Together, this proves the proposition. %\ref{prop:multiplicativityII}

\vspace{0.3 em}\noindent 
{\it Proof of (\ref{it:a}).}  
Since the correspondence with kernel $[\Delta_3]$ is just pulling back to the diagonal, this follows
from the reduction to the diagonal (see \cite[p.\,427]{Hartshorne}). This also holds for $[\Delta_3']_* \circ \times$.

\vspace{0.3 em}\noindent 
{\it Proof of (\ref{it:b}).}
Use functoriality of $\times$ (see \cite[Proposition 1.10]{Fulton}) % + compatibility with intersection
to check this explicitly.

\ifmyversion
\begin{note}
  Let $(\alpha, \beta)\in \CH(X)\times \CH(X)$. Then use the functoriality of $\times$ (see
  \cite[Proposition 1.10]{Fulton}) and its compatibility with intersection to see:
  \begin{align*}
    [Z\times Z]_* ( \alpha \times \beta )
    &={p_{X'\times X'}}_*\big((\alpha\times X'\times \beta \times X').[Z\times Z]\big)\\
    &=({p_{X'}}\times{p_{X'}})_*\Big(\big((\alpha\times X').[Z]\big)
        \times \big((\beta \times X').[Z]\big)\Big) \\
    &=p_{X'*}\big((\alpha\times X').[Z]\big) \times p_{X'*}\big((\beta\times X').[Z]\big)\\
    &=[Z]_*(\alpha)\times [Z]_*(\beta).
  \end{align*}
\end{note}
\fi

\vspace{0.3 em}\noindent 
{\it Proof of (\ref{it:c}).} We will apply similar methods as for the proof of the invertibility of $[Z]_*$.
Let $p_{124}:X\times X\times X\times X'\to X\times X\times X'$ be
the projection to the first, second, and fourth factor, and $p_{125}:X\times X \times X'\times
X' \times X' \to X\times X \times X'$ be the projection to the first, second and fifth
factor. Then
$[Z]_*\circ [\Delta_3]_*$ is the correspondence with kernel $\alpha_0:=p_{124*}\big([\Delta_3 \times
X'].[X\times X \times Z] \big)$ (see \cite[Chapter 16.1]{Fulton}), and analogously $[\Delta_3']_*\circ [Z\times
Z]_*$ is the correspondence with kernel 
$\beta_0:= p_{125*}\big([Z\times Z \times X'].[X\times X\times \Delta_3'] \big)$.

Consider families $\cX$ and $\cX'$ as in Proposition \ref{prop:families}, and keep the notation
of Section \ref{ssec:notation+theorem}.
Denote by $\overline{\Delta_{\eta 3}}$ the small diagonal in $\cX\times_T \cX\times_T \cX$ and by $\overline{\Delta_{\eta 3}'}$ the small diagonal in $\cX'\times_T \cX'\times_T \cX'$.
Then $\alpha_0$ is the specialization of the cycle $\alpha:=p_{124*}\big([\overline{\Delta_{\eta 3}} \times_T
\cX'].[\cX\times_T \cX \times_T \Gammabar] \big)$ and $\beta_0$ is the specialization of the cycle 
$\beta:=p_{125*}\big([\Gammabar\times_T \Gammabar \times_T \cX].[\cX\times_T \cX\times_T
\overline{\Delta'_{\eta 3}}] \big)$. Here, $p_{124}$ and $p_{125}$ are similar as before.

It is thus enough to show that $r_\eta(\alpha)$ coincides
with $r_\eta(\beta)$. 
Let $\Sigma$ be the image of the map 
$\cX_\eta \to \cX_\eta\times_{k(\eta)} \cX_\eta \times_{k(\eta)} \cX'_\eta$, which is induced by the
identity in the first two factors and the isomorphism $\cX_\eta \iso \cX'_\eta$ in the last factor.
Observe that:
\begin{align*}
  r_\eta(\alpha)&=p_{124*}\big([\Delta_{\eta3} \times_{k(\eta)} \cX'_\eta].[\cX_\eta\times_{k(\eta)}
  \cX_\eta \times_{k(\eta)} \Gamma] \big)\\
  &= [\Sigma] 
  = p_{125*}\big([\Gamma\times_{k(\eta)} \Gamma \times_{k(\eta)} \cX_\eta].
  [\cX_\eta \times_{k(\eta)} \cX_\eta \times_{k(\eta)}  \Delta'_{\eta3}] \big)\\ 
  &= r_\eta(\beta).
\end{align*}
\ifmyversion
In order to see that this intersection indeed gives $[\Sigma]$, observe more generally that 
$\Gamma_g\.\Gamma_f=\Gamma_{(f,g)}\big(=\{(x,f(x),g(x))\}\big)$. 
This can for example be checked on the level of affine schemes,
since the scheme-theoretic intersection of $\Gamma_f$ and $\Gamma_g$ is already the Graph of $(f,g)$.
\fi
This concludes the proof of multiplicativity and thus the proof of Theorem \ref{thm:invertibility}.
\hfill\qedsymbol

%%%%%%%%%%%%%%%%%%%%%%%%%%%%%%%%%%%%%%%%%% First applications %%%%%%%%%%%%%%%%%%%%%%%%%%%%%%%%%%%%%%%%5
%%%%%%%%%%%%%%%%%%%%%%%%%%%%%%%%%%%%%%%%%%%%%%%%%%%%%%%%%%%%%%%%%%%%%%%%%%%%%%%%%%%%%%%%%%%%%
\section{First applications} \label{sec:applications} In this section we give first applications of
Theorem \ref{thm:invertibility}: We relate it 
 to the conjectural Bloch--Beilinson filtration and apply it to Beauville's
conjecture on the weak splitting property and Voisin's generalization.

Throughout this section, we will always work with the rational Chow rings $\CH_\bQ(X)$.

%%%%%%%%%%%%%%%%%%%%%%%%%%%%%%%%%%%%%%%%%%%%%%%%%%%%%%%%%%%%%%%%%%%%%%%%%%%%%%%%%%%%%%%%%%%%%%%%%%%
%%%%%%%%%%%%%%%%%%%%%%%%%%%%%%%%%%%%%%% Application to BB-filtration %%%%%%%%%%%%%%%%%%%%%%
\subsection{Application to the Bloch--Beilinson filtration}
In contrast to the cohomology ring of a variety, the Chow ring is not well understood.
While the Hodge conjecture predicts the image of the cycle class map $$c_X:\CH_\bQ(X) \to H^*(X,\bQ),$$ its
kernel is still rather mysterious.

The Bloch--Beilinson conjecture predicts for arbitrary smooth projective $X$ the existence of a
descending filtration
 $$\CH^k_\bQ(X)=F^0\CH^k_\bQ(X)\supseteq \dotso \supseteq F^{k+1}\CH^k_\bQ(X)=0\,,$$
which is functorial, compatible with multiplication, and satisfies $F^1\CH_\bQ(X) = \ker (c_X)$.
There are various candidates for such a filtration. For a discussion of this topic we refer to
\cite[p.\,245 ff.]{Jannsen:Bloch-Beilinson-conjecture}.

Now consider birational irreducible symplectic varieties $X$ and $X'$ and a functorial candidate $F^i$ for
the filtration. As an immediate consequence of Theorem
\ref{thm:invertibility}, one observes $F^i\CH^k_\bQ(X')=[Z]_*\big(F^i\CH^k_\bQ(X)\big)$. In particular,
$F^{k+1}\CH^k_\bQ(X)=0$ if and only if $F^{k+1}\CH^k_\bQ(X')=0$.

%%%%%%%%%%%%%%%%%%%%%%%%%%%%%%%%%%%%%%%%%%%%% Application to weak splitting property %%%%%%%%%%%%%
\subsection{Application to conjectures of Beauville and Voisin}
The original motivation for this
article was to generalize \cite[Proposition 2.6]{Beauville}. In this subsection we present such a
generalization deducing it from our previous results.

Let $X$ be a nonsingular complex projective variety. In \cite{Beauville} Beauville defines
$\DCH(X)\subseteq \CH_\bQ(X)$ as the subalgebra generated by divisor classes. Then $X$ satisfies the {\it weak splitting property} if the restriction of the
  cycle class map to $$\DCH(X)\inj H^*(X,\bQ)$$ is injective.
This notion was inspired by the fact that for a simply connected $X$ the weak splitting property is satisfied, if its Bloch--Beilinson filtration {\it splits} (i.e. comes from
a ring graduation $F^p\CH^k_\bQ(X)=\bigoplus_{j=p}^k \tilde{F}^j\CH^k_\bQ(X)$).

Beauville formulates the following conjecture:
\begin{conjecture}[\cite{Beauville}] \label{conj:wsp}
  An irreducible symplectic projective variety satisfies the weak splitting property.
\end{conjecture}

A stronger version of this conjecture was formulated by Voisin in \cite{Voisin-08}:
Define $R(X)\subseteq \CH_\bQ(X)$ as the subalgebra generated by $\CH^1_\bQ(X)$ together
with $\{c_i(T_X) \}_{i\in \bN}$\,.
\begin{conjecture}[{\cite[Conjecture 1.3]{Voisin-08}}] \label{conj:Voisin}
  For any irreducible symplectic complex variety $X$ the restriction $\restr{c_X}{R(X)}:R(X)\inj H^*(X,\bQ)$ of the cycle class
  map to the subalgebra $R(X)$
   is injective.
\end{conjecture}

 For K3 surfaces $S$ the subalgebras $\DCH(S)$ and $R(S)$ coincide, and the conjectures are known to be
 true in this case (see \cite{Beauville-Voisin}). 

Furthermore, Beauville proves in \cite{Beauville} that for any K3 surface $S$ the Hilbert schemes Hilb$^2(S)$ and
Hilb$^3(S)$ satisfy the weak splitting property. In \cite{Voisin-08}, Voisin extends this result by showing
that Hilb$^n(S)$ already satisfies Conjecture \ref{conj:Voisin}, if $n\leq 2 \:b_2(S)_{tr}+4$.
Here, $b_2(S)_{tr}$ denotes the rank of the transcendental lattice of $S$.

Beauville shows (\cite[Proposition 2.6]{Beauville}) that the weak splitting property is invariant under
elementary Mukai flops (cf. Section \ref{ssec:Mukai flops}). 
By means of Theorem \ref{thm:invertibility}, we can generalize this to arbitrary birational correspondences
 on the one hand, and to the
more general conjecture of Voisin on the
other hand. At the same time, the proof below is simpler than the one in \cite{Beauville}, as it does not
make any use of the multiplicative structure of the cohomology of an irreducible symplectic variety.
\begin{theorem}\label{thm:generalize Beauville}
  Conjecture \ref{conj:wsp} and Conjecture \ref{conj:Voisin} are both invariant under birational
  correspondences;
i.e. for birational irreducible symplectic varieties $X$ and $X'$, the restriction $\restr{c_X}{\DCH(X)}$
 is injective if and only if $\restr{c_{X'}}{\DCH(X')}$ is, and the same holds for
 $\restr{c_X}{R(X)}$.
\end{theorem}

The following facts will be useful for the proof:

 \begin{lemma}\label{lem:[Z]_*(ch(T))=ch(T)}
    If $X$ and $X'$ are birational irreducible symplectic varieties, 
    then $[Z]_*\big(c_i(T_X)\big)=c_i(T_{X'})$ for
    any $i\in \bN$.
  \end{lemma}
  \begin{proof}
    Let once again $\cX$ and $\cX'$ be 
    families as in Proposition \ref{prop:families} and keep the notation of Section
    \ref{ssec:notation+theorem}. Consider the cycle
    $\alpha:=[\Gammabar]_*\big(c_i(T_{\cX|T})\big) \in \CH(\cX')$. Its restriction to the special fibre
    is $[Z]_*\big(c_i (T_{X'})\big)$. Now, one only needs to observe:
    \begin{equation*}
      \alpha_\eta:=r_\eta(\alpha)=[\Gamma]_*\big(c_i(T_{X_\eta|k(\eta)})\big)=c_i(T_{X'_\eta|k(\eta)}).
    \end{equation*}
    Application of the specialization map concludes the proof.
  \end{proof}
For later use, note that this immediately implies:  
 \begin{corollary}\label{cor:[Z]_*(td)=td}
    If $X$ and $X'$ are birational irreducible symplectic varieties, then $[Z]_*\big(\td(X)\big)=\td(X')$.
    \hfill \qedsymbol
  \end{corollary}

\begin{lemma}\label{lem:restr[Z]iso}
  Let  $X$ and $X'$ be birational irreducible symplectic varieties. Then $[Z]_*$ restricts to isomorphisms 
  $[Z]_*:\DCH(X)\overset{\iso}{\too} \DCH(X')$ and $[Z]_*:R(X)\overset{\iso}{\too} R(X')$.
\end{lemma}
\begin{proof}
  The first part is an immediate consequence of Theorem \ref{thm:invertibility}. The second  part follows
  from Lemma \ref{lem:[Z]_*(ch(T))=ch(T)}.
\end{proof}

\begin{proof}[Proof of Theorem \ref{thm:generalize Beauville}]
The isomorphisms of Lemma \ref{lem:restr[Z]iso} complete the following commutative diagram:
  $$\begin{tikzcd}[row sep = 2 ex]
    \DCH(X) \ar{r}{\iso} \ar[hook]{d}
    &\DCH(X')\ar[hook]{d}\\
    R(X) \ar{r}{\iso} \ar[hook]{d}
    &R(X')\ar[hook]{d}\\
    \CH_\bQ(X) \ar{r}{\iso}[swap]{[Z]_*} \ar{dd}[swap]{c_X}
    &\CH_\bQ(X')\ar{dd}{c_{X'}}\\
    & \\
   H^*(X,\bQ) \ar{r}{\iso}[swap]{[Z]^H_*} &H^*(X',\bQ) \ .
  \end{tikzcd}
  $$
  Theorem \ref{thm:generalize Beauville} follows immediately.
\end{proof}

% \begin{remark} 
%   Already the splitting of the Bloch--Beilinson filtration is invariant under birational correspondences
%   $X \sim X'$ birational: The Bloch--Beilinson filtration of $X$ allows for a splitting if and only if
%   the same holds for $X'$. This is an immediate consequence of the previous subsection.
% \end{remark}

%%%%%%%%%%%%%%%%%%%%%%%%%%%%%%%%%%%%%%%%%%%%%%%%%%%%%%%%%%% General facts on Chow rings %%%%%%%%%%%%%%%%%%%%%%%%%%%%%%%%%%%%%%%%%%%%%%%%%%%%%%%%%%%%%%%%%%
%%%%%%%%%%%%%%%%%%%%%%%%%%%%%%%%%%%%%%%%%%%%%%%%%%%%%%%%%%%%%%%%%%%%%%%%%%%%%%%%%%%%%%%%%%%%%%%%%%%%%%%%%%%%%%%%%%%%%%%%%%%%%%%%%%%%%%%%%%%%%%
\section{General facts on Chow rings}\label{sec:general facts}
For most of the rest of the article, we will present an alternative proof for the multiplicativity of
$[Z]_*$ in the case of general Mukai flops. In this special case, one can make explicit calculations in
the Chow rings. While it is interesting to see that it can be done, the proof is much
more intricate without the use of families as in Proposition \ref{prop:families}. This indicates that even
for these most fundamental examples of birational transforms between irreducible symplectic varieties, Theorem
\ref{thm:invertibility} is a non-trivial result.

For the convenience of the reader and to fix notations we recall some standard results on Chow rings of
projective bundles and blow-ups in this section. All results can be found in  or deduced from \cite{Fulton}.

%%%%%%%%%%%%%%%%%%%%%%%%%%%%%%%%%%%%% Chern classes of $\Omega_{\bP(F)|S}$ for projective bundles %%%%%%%%%%%%%%%%%%%%%%%%%%%%%%%%%%%%%%%%%%%%%%%
\subsection{On the Chow ring of projective bundles} \label{sec:projective bundles I}
This subsection recalls statements on the Chow rings of projective bundles. In particular, we
express the Chern classes of their relative cotangent bundle explicitly. 

Let $S$ be a nonsingular quasi-projective complex variety, $F$ a locally free sheaf on $S$, and
$\pi:\bP(F):= \PProj\big(\Sym(F\dual)\big)\to S$ the natural proper projection. 

\begin{lemma} \label{lem:c_i(Omega)}
		The Chern classes of the relative cotangent bundle of $\bP(F)$ are given by the following formula:
		$$c_i(\Omega_{\bP(F)|S}) =
		(-1)^i\sum_{j=0}^i { \binom{\rank\,(F)-j}{i-j} \pi^*(c_j(F)).\big(c_1(\dO_{\bP(F)} (1))\big)^{i-j}}.$$
\end{lemma}

\begin{proof} Use the Euler sequence:
  \begin{equation} \label{eq:Euler-sequence}
    0 \too \Omega_{\bP(F)|S}  \too \pi^*(F\dual)(-1)\too  \dO_{\bP(F)}\too 0,
  \end{equation}
  and multiplicativity of 
  Chern polynomials ($c_t(\cF):=\sum{c_i(\cF) \, t^i}$) to get
  $$ c_t(\Omega_{\bP(F)|S}) = c_t\big(\pi^*(F\dual)\otimes
  \dO_{\bP(F)} (-1)\big). $$
  Then use that for any line bundle $\cL$:
  \begin{equation}\label{eq:c_i(F otimes L)}
    c_i(\cF \otimes \cL) = \sum_{j=0}^i { \binom{\rank\,(\cF)-j}{i-j} c_j(\cF).\big(c_1(\cL)\big)^{i-j} }
  \end{equation}
(see \cite[Example 3.2.2]{Fulton}).
\end{proof}

The following fact will be needed several times.
\begin{lemma} \label{lem:pi_*(h^k)}
	Let still $\pi:\bP(F)\to S$ be a projective bundle over $S$
	 and $\dO_{\bP(F)}(1)$ be its relative $\dO(1)$ with respect to the bundle structure $F$.
	 Then:
	 \begin{equation*}
	 		\pi_*\Big(\big(c_1(\dO_{\bP(F)}(1))\big)^k\Big)=
	 		\begin{cases}
	 				0			&\text{\rm for $k<\rank(F)-1$} \\
	 				1_S			&\text{\rm for $k=\rank(F)-1$}\\
                                        -c_1(F)		&\text{\rm for $k=\rank(F)$}\hspace{1.7 em}.
	 		\end{cases}
	 \end{equation*}
\end{lemma}

\begin{proof}
  The expression in question yields the Segre classes of $F$. The lemma %\ref{lem:pi_*(h^k)}
  then follows from basic properties of Segre classes (see \cite[Proposition
  3.1.(a)]{Fulton}), and from the definition of the total Chern class as inverse of the total Segre
  class (\cite[p.\,50]{Fulton}).
\end{proof}

In the later calculation an explicit expression for still another Chern class is needed:

\begin{lemma} \label{lemma:ci(Omega-tensor-O(1))}
	The following equality holds:
	\begin{equation*} 
          c_i\big(\Omega_{\bP(F)|S}\otimes \dO_{\bP(F)}(1)\big)=\sum_{m=0}^i {(-1)^m c_1( \dO_{\bP(F)}(1))^m .\pi^*\big(c_{i-m}(F\dual)\big)}.
	\end{equation*}
\end{lemma}

\begin{proof}
  The proof is an application of \eqref{eq:c_i(F otimes L)}, Lemma \ref{lem:c_i(Omega)} and the
  following result on binomial coefficients:
\end{proof}

\begin{claim} \label{claim:binomial identity}
  Fix $r\in \bN$. For any $0\leq k\leq i \leq r$, the following identity holds:
  $$
  \sum_{j=k}^i (-1)^{j+k}\binom{r-j}{i-j}\binom{r+1-k}{j-k}=(-1)^{i+k}.
  $$
\end{claim}

\begin{proof}
  Introduce a notation for the left hand side:
  $$
  T_{i,k}^r:=\sum_{j=k}^i (-1)^{j+k}\binom{r-j}{i-j}\binom{r+1-k}{j-k}.
  $$

  Then check directly that $T_{r,k}^r= (-1)^{r+k}$. To conclude the proof it is enough to show that for
  $i<r$ the equality $T_{i+1,k+1}^r =T_{i,k}^r$ holds. This can be done by a computation, using the
  relations between binomial coefficients several times.
\end{proof}

%%%%%%%%%%%%%%%%%%%%%%%%%%%%%%%%%%%%%%%%%%%%%%%%%%%%%%%%%%%%%%%%%%%%%%%%%%%% On the Chow ring of a blow-up %%%%%%%%%%%%%%%%%%%%%%%%%%%%%%%%%
\subsection{On the Chow ring of a blow-up} \label{sec:On the Chow ring of a blow-up}
Let $X$ be a nonsingular quasi-projective complex variety. Furthermore, let $P$ be a closed
subvariety of codimension $r$ in $X$, which is also nonsingular. Let the following be the diagram of a blow-up: 
$$\begin{tikzcd}[row sep = tiny, column sep=tiny]
		 &E \ar[hook]{rr}{j} \ar{dd}[swap]{\eta} 	&	& \Xh  \ar{dd}{\phi}  \\
                 &					&\times	& \\
                 &P  \ar[hook]{rr}{i}			&	&X	\ \,.
\end{tikzcd}$$

In this situation the blow-up $\Xh$ of $X$ along $P$ is known to be nonsingular and the exceptional divisor $E$ is isomorphic to $\bP(N_{P|X})$.
Let $\cW:=\eta^*(N_{P|X})/\dO_{\bP(N_{P|X})}(-1)$.  

The following proposition provides the most important facts on the Chow ring of a blow-up. For proof see \cite[Proposition 6.7]{Fulton}. 
\begin{proposition} \label{prop:blow-up_Fulton}
With the notation introduced above:
		\begin{compactenum}[(a)]
				\item \label{part a} {
					(Key Formula). For all $\gamma \in \CH(P)$,
					$$\phi^*i_*(\gamma)=j_*\big(c_{r-1}(\cW).\eta^*(\gamma)\big)$$ 
					in $\CH(\hat{X})$.
				}
				\item	\label{part b}{				
					For all $\alpha \in \CH(X)$, $\phi_*\phi^*(\alpha)=\alpha$.
				}
				\item	\label{part c}{
					If $\eps\in \CH(E)$ and $\eta_*(\eps)=0 = j^*j_*(\eps)$, then $\eps=0$.
				}
				\item	\label{part e}{
					There is a split exact sequence of abelian groups:
					$$\begin{tikzcd}
							&0	\ar{r}	&\CH(P) \ar{r}{f} &\CH(E)\oplus\CH(X) \ar{r}{g} 	&\CH(\Xh) \ar{r}	& 0					
                                        \end{tikzcd}$$
					with $f(\gamma)=\big(c_{r-1}(\cW).\eta^*(\gamma), - i_*(\gamma)\big)$ and $g(\eps,\alpha)=j_*(\eps)+\phi^*(\alpha)$. 
					A left inverse for $f$ is given by $(\eps,\alpha) \mapsto \eta_*(\eps)$.
				}
		\end{compactenum}
\end{proposition}
Furthermore, the following lemmas will be used in later proofs.
\begin{lemma} \label{lem:eta(c_r-1(W))}
		One has $\eta_*\big(c_{r-1}(\cW)\big)=1_P$.
\end{lemma}

\begin{proof}
  Applying Lemma \ref{lemma:ci(Omega-tensor-O(1))} to $\cW \iso \big(\Omega_{\bP(N_{P|X})|P} \otimes
  \dO_{\bP(N_{P|X})} (1)\big)\dual $ yields the equality
	\begin{equation} \label{eq:c_r-1(W)}
			c_{r-1}(\cW)=(-1)^{r-1} \ \sum_{m=0}^{r-1} {(-1)^m \Big(
                          c_1\big(\dO_{\bP(N_{P|X})}(1)\big)\Big)^m
                          .\eta^*\big(c_{r-1-m}(N_{P|X}\dual)\big)} \,.
	\end{equation}
The result thus follows by Lemma \ref{lem:pi_*(h^k)}.
\end{proof}

\begin{lemma} \label{lem:existence of delta_alpha}
Let $\hat{\alpha} \in \CH(\Xh)$ with $\phi_*(\hat{\alpha})=0$. Then there exists a unique element $\eps\in\CH(E)$ satisfying $\hat{\alpha}=j_*(\eps)$ and $\eta_*(\eps)=0$. 
\end{lemma}

\begin{proof}
	By Proposition \ref{prop:blow-up_Fulton}.\ref{part e} there exist elements $\eps'\in \CH(E)$ and $\alpha\in \CH(X)$ with $\hat{\alpha}=j_*\eps'+\phi^*\alpha$.
	Applying $\phi_*$, one obtains:	
	$\phi_*(\hat{\alpha})= \phi_*(j_*\eps'+\phi^*\alpha)= \phi_*j_*\eps'+\phi_*\phi^*\alpha 
	\overset{\text{\rm \ref{prop:blow-up_Fulton}.\ref{part b}}}{=}i_*\eta_*\eps'+\alpha
	$. 

	By assumption, $\phi_*(\hat{\alpha})=0$ and therefore $\alpha=  i_*\eta_*(-\eps')$.
	One obtains:
	\begin{equation} \label{bla}
		\phi^*(\alpha)=\phi^* i_*\eta_*(-\eps')
			\overset{\text{\rm \ref{prop:blow-up_Fulton}.\ref{part a}}}{=}  j_*\Big(c_{r-1}(\cW).\eta^* \big(\eta_*(-\eps')\big)\Big).
	\end{equation}
		
	Set $\eps:=\eps'-\big(c_{r-1}(\cW).\eta^* (\eta_*(\eps'))\big)$. Then
	$\hat{\alpha}=j_*\eps'+\phi^*\alpha
	\overset{\text{\rm (\ref{bla})}}{=} j_*(\eps)
	$
	and
	$$
	\eta_*(\eps)%=\eta_*\Big(\eps'-\big(c_{r-1}(\cW).\eta^* (\eta_*(\eps'))\big)\Big) 
	\overset{\text{\rm (PF)}}{=} \eta_*(\eps')-\eta_*\big(c_{r-1}(\cW)\big).\eta_*(\eps')
	\overset{\text{\rm \ref{lem:eta(c_r-1(W))}}}{=} %\eta_*(\eps')-1_{P}.\eta_*(\eps')=
0,
	$$
        where ``(PF)'' denotes application of the projection formula, as it will always do in this article.
	Proposition \ref{prop:blow-up_Fulton}.\ref{part c} gives the uniqueness of $\eps$.
\end{proof}

%%%%%%%%%%%%%%%%%%%%%%%%%%%%%%%%%%%%%%%%%%%%%%%%%%%% Multiplicativity of [Z]_* %%%%%%%%%%%
%%%%%%%%%%%%%%%%%%%%%%%%%%%%%%%%%%%%%%%%%%%%%%%%%%%%%%%%%%%%%%%%%%%%%%%%%%%%%%%%%%%%%%%%%%%
\section{Multiplicativity of $[Z]_*$ for general Mukai flops} \label{sec:multiplicativity} 
The aim of this section is to prove in a more explicit way that for a general Mukai flop the map $[Z]_*$
is multiplicative (see Proposition \ref{prop:multiplicativity}).
In the first part of this section, we will briefly recall the construction of a Mukai flop in order to
fix the notation. After further preparation, we will
present the alternative proof of Proposition \ref{prop:multiplicativity} in Section
\ref{ssec:proof:multiplicativity}.

%%%%%%%%%%%%%%%%%%%%%%%%%%%%%%%%% Mukai flops: Notation and basic facts %%%%%%%%%%%%%%%%%%%%%%%%%%%%%%%%%%%%%%%%%%%%%%%%%%%%5
\subsection{Mukai flops: Notation and basic facts} \label{ssec:Mukai flops}
In this subsection we outline the construction of a Mukai
flop (as introduced in \cite{Mukai}) in order to
fix notation for the rest of the article. Furthermore we recall some standard facts for future reference. 

Let $X$ be an irreducible symplectic variety and $\sigma$ a non-degenerated two-form generating
$H^0(X,\Omega_X^2)$. 
Let $P \subseteq X$ be a nonsingular closed subvariety of codimension $r$, which is a $\bP^r$-bundle,
i.e. $P\iso \bP(F)\overset{\pi}{\longrightarrow} S$ for some nonsingular complex projective variety $S$ and a vector bundle $F$ of
rank $r+1$ on $S$.

   Define $\Xh$ as the blow-up of $X$ along $P$. Let $E$ denote its exceptional divisor. Set
   furthermore $P':= \bP(F\dual)$.

   Since in this situation $\Omega_{P|S}\iso N_{P|X}$, there is an isomorphism between $E$ and $\bP(\Omega_{P|S})$
   which on the other hand is isomorphic to the incidence variety $W:=\{(l,\lambda)|\,l\in \lambda \} 
         \subseteq \bP(F)\times_S \bP(F\dual)$.
   Clearly:
   \begin{lemma} \label{lem:Mukai-flop2}
     \
     \begin{compactenum}[(1)]
       \item Via the isomorphism $\bP(\Omega_{\bP(F)|S})\iso W$, the bundle \label{it:O_Omega(1) iso O(1,1)}
         $\dO_{\bP(\Omega_{\bP(F)|S})}(1)$ corresponds to  $$\pr_P^*\big(\dO_{\bP(F)}(1)\big)\otimes \pr_{P'}^*\big(\dO_{\bP(F\dual)}(1)\big).$$
       \item The normal bundle $N_{E|\Xh}$ corresponds to $\pr_P^*\big(\dO_{\bP(F)}(-1)\big)\otimes
         \pr_{P'}^*\big(\dO_{\bP(F\dual)}(-1)\big)$ via $E \iso W$. \label{it:N iso O(-1,-1)}
     \end{compactenum}
   \end{lemma}

   Since the situation is symmetric in $P$ and $P'$ one can apply \cite[Corollary 6.11]{Artin:blow-down} to see that a
   blow-down of $\Xh$ along $\bP(\Omega_{\bP(F\dual)|S})\to P'$ exists in the category of algebraic spaces.

   \begin{definition} \label{def:Mukai flop}
     If $X'$ is once again a projective variety, $X \leftarrow \Xh \to  X'$ is called a {\it general
       Mukai flop} or just {\it Mukai flop}. 
     In the special case where $P\iso \bP^n$, the triple $X \leftarrow \Xh \to  X'$ is called {\it elementary
       Mukai flop}.
   \end{definition} 
  
   In this situation $X'$ is automatically nonsingular (cf. \cite{Nakano-71} and
   \cite{Fuj-Nak-72}). Furthermore, $X'$ is again an irreducible symplectic variety and has a unique symplectic structure, which coincides
   with $\sigma$ outside $P'$.

   \begin{remark} \label{rem:4-dim connected by MF} 
     If dim$(X)=4$, elementary Mukai flops play a
     particularly important role: Any birational transform between four-dimensional irreducible
     symplectic varieties is a finite composition of elementary Mukai flops. This is a consequence of
     \cite[Theorem 2]{Wierzba}.
   \end{remark}

Throughout the rest of the article denote the natural maps as in the following diagram:
$$\begin{tikzcd}[row sep=small, column sep=tiny]
  && & &E \ar[hook]{dd}{j} \ar{dddllll}[swap]{\eta} \ar{dddrrrr}{\eta'}
  &			&		&&	\\
  &&&&&&&&				\\
  && & & \Xh \ar{dl}[swap]{\phi} \ar{dr}{\phi'}
  &			&		&&	\\
  P \ar[hook]{rrr}[swap]{i} && &X & &X' \ar[hookleftarrow]{rrr}[swap]{i'} & &&P' \, .
\end{tikzcd}$$

\begin{definition}
	Fix the following notations:
	\begin{alignat*}{3}
		&h:=c_1(\dO_{\bP(F)}(1))\hspace{5 mm} & &\in \CH(P) \ &&\big(=\CH(\bP(F))\big),	\\
		&l:=c_1(\dO_{\bP(F\dual)}(1)) & &\in \CH(P') \ &&\big(=\CH(\bP(F\dual))\big),	\\
		&H:= \eta^* (h) 		&& \in \CH(E),&&\\
		&L:=\eta'^*(l) 			&& \in \CH(E).&&
	\end{alignat*}
\end{definition}

Since the Chow ring $\CH(P)$ of the projective bundle $P\overset{\pi}{\too}S$ turns via $\pi^*$ into a free $\CH(S)$-module with
basis $1,h,..., h^r$, we can define:
\begin{definition} \label{def:sigma_k}
Let $\alpha\in \CH(X)$ and consider $i^*(\alpha)\in \CH(P)$. 
For $k=0,1,\dotso, r$ define  $\sigma_{k}^\alpha \in \CH(S)$ as the unique elements such that $$i^*(\alpha)=\sum_{k=0}^r{\pi^* (\sigma_{k}^\alpha) h^k}.$$ 
\end{definition}

For a general Mukai flop the families as in Proposition \ref{prop:families} can be chosen such that
\begin{equation}
  \label{eq:Z=}
  Z=\Xh \ \cup \ P\times_S P'.
\end{equation}
This follows from the proof of \cite[Theorem 3.4]{Huybrechts:birHK_deformations}.
From now on we will work with this $Z$.

We can now state the following proposition, which is part of the statement of Theorem
\ref{thm:invertibility}:
\begin{proposition} \label{prop:multiplicativity} 
  Let $X\leftarrow \Xh \to X'$ be a general Mukai
  flop. Then the map $[Z]_*: \CH(X) \to \CH(X')$ (as in Definition \ref{def:[Z]_*}) is multiplicative.
\end{proposition}
 Proving this by explicit computations will take up the rest of this section.

\begin{remark}
  In \cite[Lemma 2.7]{Beauville} Beauville considers an elementary Mukai flop $X\leftarrow \Xh \to
  X'$, where $P\iso \bP^r$. He states that for $\alpha \in \CH^1(X)$ the equality 
  $\big([\Xh]_*(\alpha)\big)^{r+1}=[\Xh]_*(\alpha^{r+1}) $ holds.
  Generalizing this lemma finally led to Proposition \ref{prop:multiplicativity} and later to Theorem
  \ref{thm:invertibility}. 
  Note, that already for elementary Mukai flops Proposition \ref{prop:multiplicativity} is a
  significant generalization of Beauville's Lemma.
\end{remark}

%%%%%%%%%%%%%%%%%%%%%%%%%%%%%%%%%%%%%%%%%%%%%%%%%%% Basic calculations %%%%%%%%%%%%%%%%%%%%%%%%%%%%%%%%%%%%%%%
\subsection{Basic calculations}  \label{ssec:basic calculations}
With the notation of Section \ref{ssec:Mukai flops} we will now give some lemmas which are used in the later calculations.
\begin{lemma} \label{lem:O(1)=H+L} \label{N=-H-L}
  The first Chern classes of $\dO_{\bP(\Omega_{P|S})}(1)$  and  $N_{E|\Xh}$ may be expressed in the
  following way:
  $$c_1(N_{E|\Xh})=-H-L\hspace{3 em} {\text and} \hspace{3
    em}c_1\big(\dO_{\bP(\Omega_{P|S})}(1)\big)=H+L\,.$$
  By symmetry also $c_1\big(\dO_{\bP(\Omega_{P'|S})}(1)\big)=H+L$\,.
\end{lemma}
\begin{proof}
  This is a direct corollary from Lemma \ref{lem:Mukai-flop2}.
\end{proof}

\begin{corollary}
	The class $H$ coincides with the first Chern class of the relative $\dO(1)$ on $E$, with respect
        to the bundle structure $E\iso\bP\big(\Omega_{P'|S}\otimes \dO_{\bP(F\dual)}(1)\big)$.
\hfill \qedsymbol
\end{corollary}

As a consequence, $H$ fulfils the following Chern class identity:
\begin{equation} \label{eq:chern(H)}
	\sum_{n=0}^r{\eta'^*\Big(c_{r-n}\big(\Omega_{P'|S}\otimes
          \dO_{\bP(F\dual)}(1)\big)\Big). H^n}=0 \, .
\end{equation}

\begin{lemma} \label{lem:eta_*(H^k)}
	The following identities in $\CH(P')$ hold:
	\begin{equation*}
		\eta'_*(H^k)=\eta'_*\eta^*(h^k)=
		\begin{cases}
			0					&\text{\rm for $k\leq r-2$ }	\\
			1_{X'}		&\text{\rm for $k = r-1$}	\\
			l-\pi'^*c_1 (F)	&\text{\rm for $k=r$}.
		\end{cases}
	\end{equation*}
\end{lemma}

\begin{proof}
  By Lemma \ref{lem:O(1)=H+L} one obtains $H=c_1(\dO_{\bP(\Omega_{P'|S})}(1))-L$. The statement for
  $k\leq r-2$ and $k=r-1$ then follow directly from Lemma \ref{lem:pi_*(h^k)}. For the case $k=r$ apply Lemma
  \ref{lem:c_i(Omega)} additionally, in order to determine $c_1(\Omega_{P'|S})$.
\end{proof}

Let us now prove the following more explicit form of $[Z]_*$ (for general Mukai flops):
\begin{lemma}\label{lem:[Z]_*=}
For all $\alpha \in \CH(X)$, the map $[Z]_*$ is given by the following formula: $$[Z]_*(\alpha)=\phi'_*\phi^*(\alpha)+ i'_* \pi'^* (\sigma_r^\alpha)\,,$$ 
where $\sum_{k=0}^r{\pi^*(\sigma_k^\alpha ) h^k}=i^*\alpha$ as in Definition \ref{def:sigma_k}.
\end{lemma}

\begin{proof}
Consider the two natural commutative diagrams:
	$$\begin{tikzcd}[row sep=small, column sep=small]
          & &\Xh \ar[hook]{dd}[description]{\iota_{\Xh}} \ar{dddll}[swap]{\phi} \ar{dddrr}{\phi'}
          &			&	\\
          &&&&	\\
          & & X\times X' \ar{dll}{q} \ar{drr}[swap,inner sep = 0]{q'}
          &			&\\
          X & & & &X'
	\end{tikzcd}
\hspace{2 em}{\rm and}
\hspace{2 em}
	\begin{tikzcd}[row sep=small, column sep=tiny]
          && & &P\times_S P' \ar[hook]{dd}[description]{\iota_{P\times_S P'}} \ar{dddllll}[swap]{\pr_P}
          \ar{dddrrrr}{\pr_{P'}}
          &			&		&&	\\
          &&&&&&&&				\\
          && & & X\times X' \ar{dl}{q} \ar{dr}[swap,inner sep = 0]{q'}
          &			&		&&	\\
          P \ar[hook]{rrr}[swap]{i} && &X & &X' \ar[hookleftarrow]{rrr}[swap,inner sep = 0.25 ex]{i'} &
          &&P'.
	\end{tikzcd}$$

One computes: 
\begin{align}
[Z]_*(\alpha) &= q'_*([Z].q^*\alpha) \overset{\eqref{eq:Z=}}{=}q'_*\big({\iota_{\Xh}}_* (1_{\Xh}).q^*\alpha\big)+q'_*\big({\iota_{P\times_S P'}}_* (1_{P\times_S P'}).q^*\alpha\big)   \notag \\
	&\hspace{-0.45 em}\overset{\text{\rm (PF)}}{=} 			
		\phi'_*(\phi^*\alpha) + i'_* {\pr_{P'}}_* ( \pr_P^* i^* \alpha).  	\label{[Z]_*1}
\end{align}

The diagram
$$\begin{tikzcd}[row sep= small, column sep = small]
  &P\times_S P' \ar{dl}[swap]{\pr_P} \ar{dr}{\pr_{P'}}
  & 						\\
  P \ar{dr}[swap]{\pi} & \times
  &P'	\ar{dl}[inner sep=0]{\pi'}					\\
  &S &
\end{tikzcd}$$

is a fibre product by definition. Furthermore $\pi$ and $\pi'$ are both flat and proper.
Given this situation, \cite[Proposition 1.7]{Fulton}
states that $\pr_{P'*} \pr_P^* = \pi'^*\pi_*$. 

Since furthermore by Lemma \ref{lem:pi_*(h^k)}:
\begin{equation} \label{eq:pi_*h^k}
	\pi_*h^k = 
	\begin{cases}
				1_{S} &\text{\rm if} \ k =r\\
				0			&\text{\rm if} \ k\leq r 			
	\end{cases} 
\end{equation}
 (recall that in our setting $\rank(F) = r+1$), one obtains:
\begin{align} \label{eq:pi'^*pi_*}
		{\pr_{P'}}_* ( \pr_P^* i^* \alpha) 	
		=  \pi'^*\pi_* i^* \alpha						
		\overset{\text{\rm \ref{def:sigma_k}}}{=}\sum _{k=0}^r {\pi'^*\pi_* (\pi^*\sigma_k^\alpha . h^k)  }
		\overset{\text{\rm (PF)}}{=}\sum _{k=0}^r {\pi'^*( \sigma_k^\alpha . \pi_* h^k )}=\pi'^*( \sigma_r^\alpha)\,.
\end{align}

Inserting this equality into (\ref{[Z]_*1}) yields:
\begin{equation*}
	[Z]_*(\alpha)
	\overset{\text{\rm (\ref{[Z]_*1})}}{=} \phi'_*(\phi^*\alpha) + i'_* {\pr_{P'}}_* ( \pr_P^* i^* \alpha)
	=\phi'_*\phi^*\alpha + i'_* \pi'^* (\sigma_r^\alpha)\,,
\end{equation*}
which concludes the proof.
\end{proof}

  At last, deduce from the self-intersection formula (\cite[p.\,103]{Fulton}) that for any regular embedding
  $i: P\inj X$ of codimension $r$ between complex varieties and for any $\alpha, \beta \in \CH(P)$ the
  following formula holds:
\begin{equation}\label{eq:i_*.i_*}
	i_*(\alpha).i_*(\beta)= i_*\big(\alpha.\beta.c_r(N_{P|X})\big).
\end{equation}

%%%%%%%%%%%%%%%%%%%%%%%%%%%%%%%%%%%%%%%%%%%%%%%%%%%%%%%%%%%%%%%%%%%%%%%%%%%%%%%%%%55
\subsection{Proof of Proposition \ref{prop:multiplicativity} (explicit version)}
 \label{ssec:proof:multiplicativity}
  Note that $[Z]_*$ is additive and $[Z]_*(1_X)=1_{X'}$. Since furthermore $[Z] \in \CH^{\dim X}(X\times X')$, the map $[Z]_*$
  respects the grading. It is hence enough to show multiplicativity of $[Z]_*$\,. 

	The proof relies on the explicit description of the Chow rings of projective bundles. We will show
        that the deviation from multiplicativity of $[Z]_*$ is of the form $i'_*(\gamma)$. One such
        element $\gamma$ can then be 
        determined by calculations in the Chow rings of $P$, $E$, and $P'$ and turns out to be zero. 

\subsubsection*{Preparation}
	For $\alpha\in \CH(X)$ set:
$$ \Delta_\alpha:=\phi'^*\phi'_*\phi^*(\alpha)-\phi^*(\alpha)$$ and
$$E_\alpha:=j^*\phi^*(\alpha)=\eta^*i^*(\alpha)\,.$$

	With this definition $\phi'_*(\Delta_\alpha)=(\phi'_*\phi'^*)\phi'_*\phi^*(\alpha)-\phi'_*\phi^*(\alpha)
			\overset{\text{\rm \ref{prop:blow-up_Fulton}.\ref{part b}}}{=}\phi'_*\phi^*(\alpha)-\phi'_*\phi^*(\alpha) =0$. 
	Therefore, Lemma \ref{lem:existence of delta_alpha} yields a unique element $$\delta_\alpha \in
        \CH(E), \ \ {\rm with}\ \ j_*(\delta_\alpha) =\Delta_\alpha \ {\rm and}\  \eta'_*(\delta_\alpha)=0.$$ 
	
We now start to study the map $[Z]_*$\,. 

	Let $\alpha,\beta\in \CH(X)$. By Lemma \ref{lem:[Z]_*=} we know:
	\begin {align*}
		[Z]_*\alpha.[Z]_*\beta&=\big(\phi'_*\phi^*(\alpha) + i'_*(\pi'^*\sigma_r^\alpha)\big)
                	.\big(\phi'_*\phi^*(\beta) + i'_*(\pi'^*\sigma_r^\beta\big)) \notag \\
		&=\phi'_*\phi^*(\alpha).\phi'_*\phi^*(\beta)
				+\underbrace{\phi'_*\phi^*(\alpha).i'_*(\pi'^*\sigma_r^\beta)}_{\stern}
				+\underbrace{i'_*(\pi'^*\sigma_r^\alpha).\phi'_*\phi^*(\beta)}_{\sternn}
				+\underbrace{i'_*(\pi'^*\sigma_r^\alpha).i'_*(\pi'^*\sigma_r^\beta)}_{\sternnn}.		
	\end{align*}

	Using the definition of $\Delta_\alpha$ and $\Delta_\beta$, one obtains furthermore:

	\begin{align*}
			\phi'_*\phi^*(\alpha).\phi'_*\phi^*(\beta)
			&\overset{\text{\rm \ref{prop:blow-up_Fulton}.\ref{part b}}}{=}	\phi'_*\phi'^*\Big(\phi'_*\phi^*(\alpha).\phi'_*\phi^*(\beta)\Big)
			= \phi'_*\Big((\phi^*(\alpha)+\Delta_\alpha)(\phi^*(\beta)+\Delta_\beta)\Big)			\notag \\
			&\hspace{0.45 em}= \phi'_*\big(\phi^*(\alpha.\beta)\big) + 
					\underbrace{\phi'_*\big(\Delta_\alpha.\phi^*(\beta)\big)}_{\I} +
					\underbrace{\phi'_*\big(\phi^*(\alpha).\Delta_\beta\big)}_{\II}	+
					\underbrace{\phi'_*\big(\Delta_\alpha.\Delta_\beta\big)}_{\III}	.			\notag 
	\end{align*}

	Together, this yields:
	\begin{align*}
		[Z]_*\alpha.[Z]_*\beta &= \phi'_*\phi^*(\alpha). \phi'_*\phi^*(\beta) + \stern + \sternn
                + \sternnn \phantom{\big(} \\
		&=\phi'_*\phi^*(\alpha.\beta)+ \I + \II + \III + \stern + \sternn + \sternnn \,.
	\end{align*}
	
	On the other hand, $[Z]_*(\alpha.\beta)$ is (by Lemma \ref{lem:[Z]_*=}) known to be:
	\begin{align*}
		[Z]_*(\alpha.\beta)
		=\phi'_*\phi^*(\alpha.\beta)+i'_*\pi'^*(\sigma_r^{\alpha.\beta}).
	\end{align*}
	
	In order to prove multiplicativity of $[Z]_*$ it is hence enough to show that:
	\begin{equation} \label{aim}
		\I + \II + \III + \stern + \sternn + \sternnn = i'_*\pi'^*(\sigma_r^{\alpha.\beta}).
	\end{equation}
	This will be shown by explicit calculations.
%%%%%%%%%%%%%%%%%%%%%%%%%%%%%%%%%%
\subsubsection*{First expression for $\I + \II + \III$}
	Rewrite $\I$, $\II$, and $\III$ to see that they lie in the push forward of $\CH(P')$.

	\begin{align}
		\I \hspace{0.5 em} &= \ \phi'_*\big(\Delta_\alpha.\phi^*(\beta)\big)
		= \phi'_*\big(j_*(\delta_\alpha).\phi^*(\beta)\big)
		\overset{\text{\rm (PF)}}{=}
                \phi'_*j_*\big(\delta_\alpha.\underbrace{j^*\phi^*(\beta)}_{=E_\beta}\big)      \notag   \\
		&\hspace{-0.1 em} \overset{\text{\rm \ref{def:sigma_k}}}{=} i'_*\eta'_*\Big(\delta_\alpha.\big(\sum_{i=0}^r {\eta'^*\pi'^*\sigma_i^\beta.H^i}\big)\Big) 
		\overset{\text{\rm (PF)}}{=} i'_*\big(\sum_{i=0}^r {\pi'^*\sigma_i^\beta . \eta'_*(\delta_\alpha.H^i)}\big).                   \label{I.b}
	\end{align}

	Switching the roles of $\alpha$ and $\beta$, one obtains:
	$\II = \phi'_* j_* (E_\alpha.\delta_\beta) $.
	
	\begin{align*}
		\III \ 
&= \ \phi'_*(\Delta_\alpha.\Delta_\beta)
=\phi'_*\Big(\big(\phi'^*\phi'_*\phi^*(\alpha)-\phi^*(\alpha)\big).j_*(\delta_\beta)\Big) \\
 &\hspace{-1.5 mm} \overset{\text{\rm (PF)}}{=}
\phi'_*\phi^*(\alpha) .\underbrace{\phi'_*\underbrace{j_*(\delta_\beta)}_{\Delta_\beta}}_{=0}
- \phi'_* j_* \big(j^*\phi^*(\alpha).\delta_\beta\big)
=- \phi'_* j_* (E_\alpha.\delta_\beta) = -\II\,.
	\end{align*}
	
	Hence one obtains:
        \begin{equation} \label{I+II+III}
		\I + \II + \III= \ \ \I 
                \overset{\text{\rm \eqref{I.b}}}{=} i'_*\big(\sum_{i=0}^r {\pi'^*\sigma_i^\beta . \eta'_*(\delta_\alpha.H^i)}\big). 
	\end{equation}
	
Note that this expression consist of terms of the form  $i'_*\big(\xi. \eta'_*(\delta_\alpha.H^i)\big)$, with $\xi \in\CH(P')$\,. 
	
\subsubsection*{Expression for $i'_*\big(\xi. \eta'_*(\delta_\alpha.H^i)\big)$ independent of $\delta_\alpha$}
	Fix a class $\xi\in\CH(P')$ and $0 \leq j \leq r$.  The expression
        $\phi'_*\big(\Delta_\alpha.j_*(\eta'^*\xi. H^j)\big)$ can be rewritten in two different ways. 
	
	On the one hand:
	\begin{align*}
		\phi'_*\big(\Delta_\alpha.j_*(\eta'^*\xi . H^j )\big)
		&=\phi'_*\Big(\big(\phi'^*\phi'_*\phi^*(\alpha)-\phi^*(\alpha)\big).j_*(\eta'^*\xi . H^j )\Big) 		\notag \\
		&\hspace{-0.45 em} \overset{\text{\rm (PF)}}{=} \phi'_*\phi^*(\alpha).\phi'_*j_*(\eta'^*\xi . H^j)-\phi'_*j_*\big(j^*\phi^*(\alpha).\eta'^*\xi. H^j \big)											\notag \\
		&= \phi'_*\phi^*(\alpha).i'_*\eta'_*(\eta'^*\xi . H^j)
							-i'_*\eta'_*(E_\alpha.\eta'^*\xi . H^j )	 		\notag \\
		&\hspace{-0.45 em}\overset{\text{\rm (PF)}}{=} \ \phi'_*\phi^*(\alpha).i'_*\big(\xi. \underbrace{\eta'_*(H^j)}_{=0 \ \forall j \leq r-2}\big)
						-i'_*\big(\xi .\eta'_*(E_\alpha. H^j)\big).											
	\end{align*} 

 	On the other hand using \eqref{eq:i_*.i_*} and Lemma \ref{N=-H-L} yields:
 	\begin{align*}
 		\phi'_*\big(\Delta_\alpha.j_*(\eta'^*\xi . H^j )\big)
 		&=\phi'_*\big(j_*(\delta_\alpha).j_*(\eta'^*\xi . H^j)\big)  
		=i'_*\Big(\xi.\eta'_*\big(\delta_\alpha. H^j (-H-L)\big)	\Big).
 	\end{align*}

	Together this shows:
	\begin{equation} \label{delta_alpha2}
			i'_*\Big(\xi.\eta'_*\big(\delta_\alpha. H^j (H+L)\big)	\Big)
                        =i'_*\big(\xi .\eta'_*(E_\alpha. H^j)\big), \hspace{10 mm} \forall\ 0 \leq j \leq r-2.
	\end{equation}
	
	Furthermore, for $j=r-1$ the above equality becomes (using $\eta'_*(H^{r-1})=1_{P'}$):
	\begin{equation} \label{eq:phi'_*phi^*(alpha)}
			\phi'_*\phi^*(\alpha).i'_*\xi=i'_*\big(\xi.\eta'_*(E_\alpha.H^{r-1})\big)-i'_*\Big(\xi.\eta'_*\big(\delta_\alpha.H^{r-1}(H+L)\big)\Big).
	\end{equation}
	
	By means of \eqref{delta_alpha2} we now prove:
	
	\begin{claim} \label{claim:delta_alpha}
			For all $0 \leq j\leq r-1$ and for all $\xi \in \CH(P')$ the following equation holds:
			\begin{equation*}
				i'_*\big(\xi.\eta'_*(\delta_\alpha. H^j)\big)= \sum_{i=0}^{j-1}{(-1)^i \ i'_*\big(\xi.l^{i}.\eta'_*(E_\alpha . H^{j-i-1})\big)}.
			\end{equation*}
	\end{claim}
	
	\begin{proof}  In the case $j=0$, it is
          enough to recall that $\eta'_*(\delta_\alpha)=0$ by the choice of $\delta_\alpha$. 
          The full statement then follows inductively from the equation
		$$i'_*\big(\xi.\eta'_*(\delta_\alpha. H^{j+1})	\big)+i'_*\big(\xi.l.\eta'_*(\delta_\alpha. H^j)	\big)
			\overset{\text{\rm (PF)}}{=} 
				i'_*\big(\xi.\eta'_*(\delta_\alpha. H^j (H+L))	\big)
			\overset{\text{\rm \eqref{delta_alpha2}}}{=} i'_*\big(\xi .\eta'_*(E_\alpha. H^j)\big)$$
which holds for $0\leq j\leq r-2$.
	\end{proof}

\subsubsection*{First expressions for $\stern$, $\sternn$ and $\sternnn$}	
	We now rewrite the terms $\stern$, $\sternn$
        and $\sternnn$. The new forms will show that they are in the image of $i'_*:\CH(P')\to \CH(X')$. 

	Use  \eqref{eq:phi'_*phi^*(alpha)} and Claim \ref{claim:delta_alpha} to see:	
	\begin{align}
          \stern &= \phi'_*\phi^*(\alpha).i'_*(\pi'^*\sigma_r^\beta)  \notag\\
          &\hspace{-0.65 em} \overset{{\rm \eqref{eq:phi'_*phi^*(alpha)}}}{=} 
             \ i'_*\Big(\pi'^*\sigma_r^\beta.\eta'_*(E_\alpha.H^{r-1})
            -\pi'^*\sigma_r^\beta.l.\eta'_*(\delta_\alpha.H^{r-1})
            -\pi'^*\sigma_r^\beta.\eta'_*(\delta_\alpha.H^{r})\Big)	\notag \displaybreak[0]\\
          &\hspace{-0.3 em} \overset{\text{\rm
              \ref{claim:delta_alpha}}}{=}i'_*\left(\pi'^*\sigma_r^\beta.\eta'_*(E_\alpha.H^{r-1})
            -\pi'^*\sigma_r^\beta.\sum_{i=0}^{r-2}{(-1)^i. l^{i+1}.\eta'_*(E_\alpha.H^{r-1-i-1})}
            -\pi'^*\sigma_r^\beta.\eta'_*(\delta_\alpha.H^{r})\right)	\notag \displaybreak[0]\\
          &=i'_*\left(\pi'^*\sigma_r^\beta.\sum_{i=0}^{r-1}{(-1)^i. l^{i}.\eta'_*(E_\alpha.H^{r-i-1})}
            -\pi'^*\sigma_r^\beta.\eta'_*(\delta_\alpha.H^{r})\right).  \label{stern2}
	\end{align}
	Switching $\alpha$ and $\beta$, one obtains the following expression for $\sternn$:
        \begin{equation} \label{sternn1}
          \sternn \ 
          =i'_*\bigg(\pi'^*\sigma_r^\alpha
          \sum_{i=0}^{r-1}{(-1)^i. l^{i}.\eta'_*(E_\beta.H^{r-i-1})}
          -\pi'^*\sigma_r^\alpha.\eta'_*(\delta_\beta.H^{r})\bigg).
	\end{equation}

	For $\sternnn$, application of \eqref{eq:i_*.i_*} gives:
	\begin{equation} \label{eq:sternnn1}
		\sternnn= 	i'_*(\pi'^*\sigma_r^\alpha).i'_*(\pi'^*\sigma_r^\beta) = i'_*\big(\pi'^*(\sigma_r^\alpha.\sigma_r^\beta).c_r(\Omega_{P'|S})\big).
	\end{equation}

\subsubsection*{Notation: $\tau_{i,j}$}

	Apply the fact that via $\pi^*$ the ring
	$\CH(P)$ is  a free $\CH(S)$-module generated by $1, h, h^2,... h^r$, to make
        the following definition: 
	\begin{definition}\label{def:tau_i,j}
		Let $i\in \bN_0$. Define $\tau_{i,j} \in \CH(S)$ for $j=0,1,\dotso, r$ as the unique elements such that:
		$$h^i=\sum_{j=0}^r{\pi^*(\tau_{i,j}).  h^j}\,.$$
		Furthermore, set $\tau_{i,j}=0$ for all $j\notin \{0,1,\dotso, r\}$.
	\end{definition}
	
	\begin{remark} \label{rem:tau=delta}
			Note that with this definition $\tau_{i,j}=\delta_{i,j}.1_S$ for all $i\leq r$, where $\delta_{i,j}$ is the Kronecker delta.
	\end{remark}
	
	\begin{remark} \label{rem:tau_i,j2}
			There is a recursive relation between the $\tau_{i,j}$ given by:
			$$\tau_{i+1, j} =\tau_{i,j-1} - c_{r+1-j}(F)\ \tau_{i,r} \,.$$
	\end{remark}
        \begin{proof}[Proof of Remark \ref{rem:tau_i,j2}]
          The statement follows by comparing the coefficients in:
          \begin{align*}
            h^{i+1}=h.h^i
            =\sum_{j=0}^r{\pi^*(\tau_{i,j}).  h^{j+1}} 
            =\sum_{j=1}^r{\pi^*(\tau_{i,j-1}).  h^{j}} - \sum_{j=0}^r{\pi^*\big(\tau_{i,r} \
              c_{r+1-j}(F)\big) \ h^j }.
          \end{align*}
        \end{proof}

\subsubsection*{Final expression for $i'_*\big(\pi'^*(\sigma_r^{\alpha.\beta})\big)$}
With this notation, we can express the right hand side of \eqref{aim} more explicitly:
	\begin{align}
			i'_*\big(\pi'^*(&\sigma_r^{\alpha.\beta})\big)
			\overset{\text{\rm \eqref{eq:pi'^*pi_*}}}{=} i'_*\pi'^*\big(\pi_*(i^*(\alpha.\beta))\big)
			= i'_*\pi'^*\pi_*\Big(\big(\sum_{k=0}^r{\pi^*(\sigma_k^\alpha).h^k}\big)\big(\sum_{j=0}^r{\pi^*(\sigma_j^\beta).h^j}\big)\Big)	\notag \\
			&=i'_*\pi'^*\Big(\sum_{k=0}^r{\sum_{j=0}^r{\sigma_k^\alpha\sigma_j^\beta\sum_{m=0}^r{\tau_{k+j,m}.\pi_*(h^m)}}} \Big)	
			\overset{\text{\rm \eqref{eq:pi_*h^k}}}{=} 		
					i'_*\pi'^*\Big(\sum_{k=0}^r{\sum_{j=0}^r{\sigma_k^\alpha\sigma_j^\beta.\tau_{k+j,r}}} \Big).  \label{eq:RHS(aim)_final}
	\end{align}	
	
In the sequel, we study the summands on	the left hand side of \eqref{aim}.

\subsubsection*{Final expression for $\I + \II + \III + \stern$}
	From \eqref{stern2} together with \eqref{I+II+III} one obtains:
	\begin{align}
		\I +& \II + \III + \stern \notag \\
		&= i'_*\left(\sum_{j=0}^r {\pi'^*\sigma_j^\beta . \eta'_*(\delta_\alpha.H^j)}\right)
				+ i'_*\left(\pi'^*\sigma_r^\beta \sum_{i=0}^{r-1}{(-1)^i. l^{i}.\eta'_*(E_\alpha.H^{r-i-1})}
				-\pi'^*\sigma_r^\beta.\eta'_*(\delta_\alpha.H^{r})\right)	\notag  \\
		&= i'_*\left(\sum_{j=0}^{r-1} {\pi'^*\sigma_j^\beta . \eta'_*(\delta_\alpha.H^j)}
				+ \pi'^*\sigma_r^\beta \sum_{i=0}^{r-1}{(-1)^i. l^{i}.\eta'_*(E_\alpha.H^{r-i-1})}		\right)	\notag   \displaybreak[0]\\
		&\hspace{-0.3 em} \overset{\rm \ref{claim:delta_alpha}}{=}
				i'_*\left(\sum_{j=0}^{r-1} {\pi'^*\sigma_j^\beta \sum_{i=0}^{j-1}{(-1)^i.l^i. \eta'_*(E_\alpha.H^{j-i-1})}}
				+\pi'^*\sigma_r^\beta\sum_{i=0}^{r-1}{(-1)^i. l^{i}.\eta'_*(E_\alpha.H^{r-i-1})}			\right)		\notag  \displaybreak[0] \\
		&= i'_*\left(\sum_{j=0}^{r} {\pi'^*\sigma_j^\beta \sum_{i=0}^{j-1}{(-1)^i.l^i. \eta'_*(E_\alpha.H^{j-i-1})}}	\right)	\notag  \\
		&\hspace{-0.1 em} \overset{\rm \ref{def:sigma_k} }{=}
				i'_*\left(\sum_{k=0}^r{\sum_{j=0}^{r} {\pi'^*(\sigma_k^\alpha.\sigma_j^\beta) \sum_{i=0}^{j-1}{(-1)^i.l^i. \eta'_*(H^{k+j-i-1})}}}\right). \label{Term1.1}
	\end{align}
	
	The following claim computes the sum appearing in the last line:
	\begin{claim} \label{claim:help-sum}
		For all $j, k \in \bN_0$ the following equation holds:
			$$ \sum_{i=0}^{j-1}{(-1)^i l^i \eta'_*(H^{k+j-i-1})} = \pi'^*(\tau_{k+j,r})+ (-1)^{j-1} l^j . \pi'^*(\tau_{k,r}).$$ 
	\end{claim}
	\begin{proof}
		Using the definition of the $\tau_{i,j}$,  Lemma \ref{lem:eta_*(H^k)}  and Remark \ref{rem:tau_i,j2}, one computes:
		\begin{align*}
			\sum_{i=0}^{j-1}{(-1)^i l^i \eta'_*(H^{k+j-i-1})}
			&=\sum_{i=0}^{j-1}{(-1)^i l^i \eta'_*(\sum_{m=0}^r{\eta'^*\pi'^*(\tau_{k+j-i-1,m}).H^m})}  \\
			&=\sum_{i=0}^{j-1}{(-1)^i l^i \Big(\pi'^*(\tau_{k+j-i-1,r-1}) + \pi'^*(\tau_{k+j-i-1,r})\big(l-\pi'^*c_1(F)\big)\Big) } \displaybreak[0] \\
			&=\sum_{i=0}^{j-1}{(-1)^i l^i \big(\pi'^*(\tau_{k+j-i,r}) +l. \pi'^*(\tau_{k+j-i-1,r})\big) } \\
			&=\pi'^*(\tau_{k+j,r}) + (-1)^{j-1} l^j . \pi'^*(\tau_{k,r}).
		\end{align*}
                This proves the claim.
	\end{proof}
	Using this claim, one can simplify \eqref{Term1.1} in the following way:
	\begin{align} 
		\I + \II + \III & + \stern
		\overset{\rm \eqref{Term1.1}}{=}
		i'_*\bigg(\sum_{k=0}^r{\sum_{j=0}^{r} {\pi'^*(\sigma_k^\alpha.\sigma_j^\beta) \sum_{i=0}^{j-1}{(-1)^i.l^i. \eta'_*(H^{k+j-i-1})}}}\bigg) \notag \\
		&= i'_*\bigg(\sum_{k=0}^r{\sum_{j=0}^{r} {\pi'^*(\sigma_k^\alpha.\sigma_j^\beta)
					.\Big(\pi'^*(\tau_{k+j,r}) + (-1)^{j-1} l^j . \pi'^*(\tau_{k,r})\Big) }}\bigg) \notag \\
		&\hspace{-0.3 em }\overset{\rm \ref{rem:tau=delta}}{=} 
			i'_*\bigg(\sum_{k=0}^r{\sum_{j=0}^{r} {\pi'^*(\sigma_k^\alpha\sigma_j^\beta.\tau_{k+j,r})}
				+ \sum_{j=0}^{r}{\pi'^*(\sigma_r^\alpha.\sigma_j^\beta).(-1)^{j-1} l^j} }\bigg).						\label{eq:I+II+III+stern_final} 
	\end{align}
Note that the first part of this term coincides with the expression obtained for 
$i'_*\big(\pi'^*(\sigma_r^{\alpha.\beta})\big)$ in \eqref{eq:RHS(aim)_final}.

\subsubsection*{Final expression for $\sternn$}
	Recall that \eqref{sternn1} gives:
        \begin{equation*}
          \sternn \ 
          =i'_*\bigg(\pi'^*\sigma_r^\alpha
          \sum_{i=0}^{r-1}{(-1)^i. l^{i}.\eta'_*(E_\beta.H^{r-i-1})}
          -\pi'^*\sigma_r^\alpha.\eta'_*(\delta_\beta.H^{r})\bigg).
	\end{equation*}
        Let us first rewrite the second summand separately. Note that Claim \ref{claim:delta_alpha} only
        holds for $H^n$ with $n\leq r-1$. In order to reduce to this case, apply the Chern class
        identity \eqref{eq:chern(H)}.
        To shorten notation, set $\ell:=\dO_{\bP(F')}(1)$. With this definition $c_1(\ell)=l$. Then
        combining Claim \ref{claim:delta_alpha} and \eqref{eq:chern(H)} one obtains:
	\begin{align*}
			i'_*\bigg(-\pi'^*\sigma_r^\alpha. & \eta'_*(\delta_\beta.H^{r})\bigg)
				=
				i'_*\bigg(\pi'^*\sigma_r^\alpha \sum_{n=0}^{r-1}{c_{r-n}(\Omega_{P'|S}\otimes \ell)
											\sum_{i=0}^{n-1}{(-1)^i.l^{i}.\eta'_*(E_\beta . H^{n-i-1})}}\bigg). 
	\end{align*}
	The left summand of \eqref{sternn1} corresponds to the ($n=r$)\,-\,part of this sum. Hence:
	\begin{align*}
			\sternn \ 
			&=i'_*\bigg(\pi'^*\sigma_r^\alpha \sum_{n=0}^{r}{c_{r-n}(\Omega_{P'|S}\otimes \ell)
											\sum_{i=0}^{n-1}{(-1)^i.l^{i}.\eta'_*(E_\beta . H^{n-i-1})}}\bigg) \\
			&\hspace{-0.1 em} \overset{\rm \ref{def:sigma_k}}{=}
				i'_*\bigg(\sum_{j=0}^r{\pi'^*(\sigma_r^\alpha.\sigma_j^\beta) \sum_{n=0}^{r}{c_{r-n}(\Omega_{P'|S}\otimes \ell)
											\sum_{i=0}^{n-1}{(-1)^i.l^{i}.\eta'_*(H^{j+n-i-1})}}}\bigg) \displaybreak[0]\\
			&\hspace{-0.3 em} \overset{\rm  \ref{claim:help-sum}}{=}
				i'_*\bigg(\sum_{j=0}^r{\pi'^*(\sigma_r^\alpha.\sigma_j^\beta) \sum_{n=0}^{r}{c_{r-n}(\Omega_{P'|S}\otimes \ell)
											\Big(\pi'^*(\tau_{n+j,r})+ (-1)^{n-1} l^n . \pi'^*(\tau_{j,r})\Big)  }}\bigg)\displaybreak[0] \\
			&\hspace{-0.3 em} \overset{\rm \ref{rem:tau=delta}}{=}
				i'_*\bigg(\sum_{j=0}^r{\pi'^*(\sigma_r^\alpha.\sigma_j^\beta)   \underbrace{ \sum_{n=0}^{r}{c_{r-n}(\Omega_{P'|S}\otimes \ell)
											.\pi'^*(\tau_{n+j,r})  }}_{=:T_1(j)}	}\\
					&\hspace{15em}	+\pi'^*(\sigma_r^\alpha.\sigma_r^\beta) 		\underbrace{\sum_{n=0}^{r}{(-1)^{n-1} l^n.c_{r-n}(\Omega_{P'|S}\otimes \ell)}}_{=:T_2}
											   		\bigg).
	\end{align*}
	
	We need to develop more explicit expressions for $T_1(j)$ and $T_2$.
	\begin{claim}	
	For all $j=0, 1, \dotso, r$ the identity
	$$T_1(j)=(-1)^j l^j $$ holds.
	\end{claim}	
	
	\begin{proof}
	
		In fact, we will prove a slightly more general statement. 
		We prove that for any $j,q\in \{0,1,\dotso, r\}$ the following equality holds: 
		\begin{equation} \label{eq:aimofclaim}
			T_1(j,r-q):=\sum_{n=0}^{r}{c_{r-n}(\Omega_{P'|S}\otimes
                          \ell).\pi'^*(\tau_{n+j,r-q})  }
                        = (-1)^j l^j \Big( \sum_{m=0}^q {(-1)^m l^m .\pi'^*\big(c_{q-m}(F)\big)} \Big).
		\end{equation}
		
		Note that $T_1(j,r)=T_1(j)$. Therefore, the statement of the claim follows from \eqref{eq:aimofclaim} by setting $q=0$. 

                We prove \eqref{eq:aimofclaim} by induction on $j$.
		For $j=0$ one simply observes:
		\begin{equation*}	
			T_1(0,r-q)=\sum_{n=0}^{r}{c_{r-n}(\Omega_{P'|S}\otimes \ell).\pi'^*(\tau_{n,r-q})  }
			\overset{\rm \ref{rem:tau=delta}}{=}c_{q}(\Omega_{P'|S}\otimes \ell) 
			\overset{\rm \ref{lemma:ci(Omega-tensor-O(1))}}{=}\sum_{m=0}^q {(-1)^m l^m .\pi'^*\big(c_{q-m}(F)\big)}.
		\end{equation*}

                For every $j\leq r-1$:
		\begin{align*}
			T_1(j+1,r-q)
                        \hspace{0.3 em} & \hspace{-0.3 em}\overset{\rm \ref{rem:tau_i,j2}}{=}
                        T_1\big(j,r-(q+1)\big)- \pi'^*\big(c_{q+1}(F)\big). T_1(j,r).
		\end{align*}
                Then application of the induction hypothesis concludes the proof of the claim.
	\end{proof}
		
	Now determine $T_2$:
	\begin{align*}
	T_2&=\sum_{n=0}^{r}{(-1)^{n-1} l^n.c_{r-n}(\Omega_{P'|S}\otimes \ell)}
		\overset{\rm \ref{lemma:ci(Omega-tensor-O(1))}}{=} 
				\sum_{n=0}^{r}{(-1)^{n-1} l^n \sum_{m=0}^{r-n} {(-1)^m l^m .\pi'^*\big(c_{r-n-m}(F)\big)}}\\
				&=\sum_{m'=0}^{r} {(-1)^{m'-1}(m'+1). l^{m'}
                                  .\pi'^*\big(c_{r-m'}(F)\big)}.
	\end{align*}
	
	Using the expressions for $T_1(j)$ and $T_2$, one obtains:
	\begin{align} \label{eq:sternn_final}
		\sternn
		&= i'_*\bigg(\sum_{j=0}^r{\pi'^*(\sigma_r^\alpha.\sigma_j^\beta) .  (-1)^j l^j	}\notag\\
						&\hspace{6 em}+\pi'^*(\sigma_r^\alpha.\sigma_r^\beta) 	.\sum_{m'=0}^{r} {(-1)^{m'-1}(m'+1). l^{m'} .\pi'^*\big(c_{r-m'}(F)\big)}
											   		\bigg).
	\end{align}

\subsubsection*{Final expression for $\sternnn$}
 Recall that by \eqref{eq:sternnn1}
	\begin{equation*}
			 \sternnn= i'_*\big(\pi'^*(\sigma_r^\alpha.\sigma_r^\beta).c_r(\Omega_{P'|S})\big).
	\end{equation*}
	
	Lemma \ref{lem:c_i(Omega)} implies:
	\begin{align*}
			c_r(\Omega_{P'|S}) 
		  &=\sum_{m'=0}^r {(-1)^{m'} (m'+1). \pi'^*\big(c_{r-m'}(F)\big).l^{m'}}.
	\end{align*} 
	
	Hence $\sternnn$ can be expressed as:
	\begin{equation} \label{eq:sternnn_final}
		\sternnn = i'_*\bigg(\pi'^*(\sigma_r^\alpha.\sigma_r^\beta).\sum_{m'=0}^r {(-1)^{m'} (m'+1). \pi'^*\big(c_{r-m'}(F)\big).l^{m'}}\bigg).
	\end{equation}
	
\subsubsection*{Conclusion of the proof}
	Now combine the calculations of $\I + \II + \III + \stern$, $\sternn$, and $\sternnn$
	 given in   \eqref{eq:I+II+III+stern_final}, \eqref{eq:sternn_final}, and
         \eqref{eq:sternnn_final} respectively, to obtain:
	\begin{align*}
			\I + &\II + \III + \stern + \sternn + \sternnn
			 =i'_*\bigg(\sum_{k=0}^r{\sum_{j=0}^{r}
                           {\pi'^*(\sigma_k^\alpha\sigma_j^\beta.\tau_{k+j,r})}} \bigg) \overset{\rm \eqref{eq:RHS(aim)_final}}{=} i'_*\big(\pi'^*(\sigma_r^{\alpha.\beta})\big). \notag
	\end{align*}
	
	By \eqref{aim}, this concludes the (alternative) proof of Proposition
        \ref{prop:multiplicativity}.
	
\hfill \qedsymbol

%%%%%%%%%%%%%%%%%%%%%%%%%%%%%%%%%%%%%%%%%%%%%%%%%%%%%%%%%%%%%%%%%%%%%%%%%%%%%%%%%%%%%%%%%%%
%%%%%%%%%%%%%%%%%%%%%%%%%%%%%%%%%%%%%%%%%%%%%%%%%%%%%%%%%%%%%%%%%%%%%%%%%%%%%%%%%%%%%%%%%
\section{Link to derived categories and Grothendieck groups} \label{subsec:link to D}
This section relates Theorem \ref{thm:invertibility}
with a result of Namikawa on derived categories (\cite[Theorem 5.1]{Namikawa2003}). Finally, we discuss
the consequence of Theorem \ref{thm:invertibility} on the level of Grothendieck groups.
\bigskip

Denote by $\rD^b(X):=\rD^b\big({\bf Coh}(X)\big)$ the bounded derived category of an algebraic variety $X$. For
the notation and basic results on derived categories we refer to \cite{Huybrechts:FM-transforms}.
 
Consider a general Mukai flop $X\leftarrow \Xh \to X'$. Denote by $\iota_Z$ the natural inclusion of $Z$
into $X\times X'$ and by $q$ and
$q'$ the projections from $X\times X'$ to $X$ and $X'$ respectively.

In \cite{Namikawa2003} Namikawa studies the functor $\Phi_{\dO_Z}:=R(q'\circ\iota_Z)_*\circ L(q\circ\iota_Z)^*
: \rD^b(X) \to \rD^b(X')$,
which coincides with the Fourier--Mukai transform with kernel $\dO_Z$. He proves the following result,
which for elementary Mukai flops is due to Kawamata (see \cite[Corollary 5.7]{Kawamata02}):
\begin{theorem} [{\cite[Theorem 5.1]{Namikawa2003}}]\label{thm:Namikawa}
  For a general Mukai flop $X\leftarrow \Xh \to X'$, the functor $\Phi_{\dO_Z}$ is an equivalence of
  triangulated categories.
\end{theorem}

For arbitrary birational irreducible symplectic varieties $X$ and $X'$, the natural maps \linebreak[0]
$\rD^b(X) \overset{[\ ]}{\too}\linebreak[0] K(X)\overset{v}{\too}\CH_\bQ(X)$ (where  $v$
associates to a coherent sheaf its Mukai vector $v(F):=\ch(F).\sqrt{\td(X)}\in \CH_\bQ(X)$) give rise to
the following commutative diagram:

\begin{equation} \label{diag:FM-diagram}
  \begin{tikzcd}
    \rD^b(X)\dar[swap]{[\ ]} \rar{\Phi_{\dO_Z}}&\rD^b(X')\dar{[\ ]}\\
    K(X)\dar[swap]{v}  \rar{\Phi^K_{[\dO_Z]}}     &K(X') \dar{v}\\
    \CH_\bQ(X) \rar{\Phi^{\CH}_{v(\dO_Z)}} &\CH_\bQ(X').
  \end{tikzcd}
\end{equation}
Note that Theorem \ref{thm:Namikawa} implies bijectivity of the map $\Phi^{\CH}_{v(\dO_Z)}$ in the case
of general Mukai flops.

The central result of this section is:
\begin{proposition} \label{prop:[Z]=v(Oz)} Let $X$ and $X'$ be birational irreducible symplectic
  varieties. Then the classes
  $v(\dO_Z)$ and $[Z] \in \CH_\bQ(X\times X')$ coincide. Consequently, the associated correspondences $
  \Phi^{\CH}_{v(\dO_Z)} =[Z]_*$ are equal. In particular, $ \Phi^{\CH}_{v(\dO_Z)}$ is multiplicative.
\end{proposition}
\begin{proof}
Let $\cX$ and $\cX'$ be families as in Proposition \ref{prop:families}.
Consider the following two cycles in $\CH_\bQ(\cX\times_T\cX')$:
\begin{enumerate}[(1)]
\item The cycle $\alpha:=[\Gammabar]$, and
\item the cycle $\beta:=\ch(\dO_{\Gammabar}).\sqrt{\td(T_{\cX\times_T\cX'|T})}$.
\end{enumerate}

As a consequence of Proposition \ref{prop:families}, the restriction of $\alpha$ to the special fibre is 
$s_0(\alpha)=[Z]$. On the other hand, compute:   \begin{align*}
    s_0(\beta)
    &= s_0 \Big( \ch(\dO_{\Gammabar}).\sqrt{\td(T_{\cX\times_T\cX'|T})}\Big)
    = \ch(\restr{\dO_{\Gammabar}}{\cX_0}).\sqrt{\td(\restr{T_{\cX\times_T\cX'|T}}{\cX_0})}\\
    &= \ch(\dO_Z).\sqrt{\td(T_{X\times X'})}
    =v(\dO_Z).
  \end{align*}

  Therefore, applying the specialization map of Section \ref{sec:specialization}, it is enough to show that
  the restrictions to the general fibre, $\alpha_\eta:= r_\eta(\alpha)$ and $\beta_\eta:=r_\eta(\beta)$, coincide.

  Consider the graph $\Gamma$ of the isomorphism $\cX_\eta \iso \cX'_\eta $, which is the restriction of
  $\cX_{T\setminus \{0\}} \iso \cX'_{T\setminus \{0\}} $. Then
  $\alpha_\eta=r_\eta([\Gammabar])=[\Gammabar_\eta]=[\Gamma]$.
  Furthermore, one computes:
  \begin{align*}
    \beta_\eta=r_\eta\big(\ch(\dO_{\Gammabar}).\sqrt{\td(T_{\cX\times_T\cX'|T})}\big) 
    %=\ch(i^*_\eta\dO_{\Gammabar}).\sqrt{\td(i^*_\eta T_{\cX\times_T\cX'|T})}
    =\ch(\dO_{\Gamma}).\sqrt{\td(T_{\cX_\eta\times_{k(\eta)}\cX'_\eta|k(\eta)})}.
  \end{align*}
  % where the last identity follows from base change of differentials%  (see \cite[Proposition
  % II.8.10]{Hartshorne})

  Let $i_\Gamma: \Gamma \to \cX_\eta\times_{k(\eta)}\cX'_\eta$ denote the natural inclusion.
  Applying the Grothendieck--Riemann--Roch theorem (GRR) (see \cite[Theorem 15.2]{Fulton}) one then observes:
  \begin{align*}
    \beta_\eta
    &=\ch({i_\Gamma}_*\dO_{\Gamma}).\td(T_{\cX_\eta\times_{k(\eta)}\cX'_\eta|k(\eta)})
    .\frac{1}{\sqrt{\td(T_{\cX_\eta\times_{k(\eta)}\cX'_\eta|k(\eta)})}}\\
    &\hspace{-0.8 em}\overset{\rm (GRR)}{=}{i_\Gamma}_*\Big(
    \underbrace{\ch(\dO_{\Gamma})}_{=1}
    .\td(T_{\Gamma|k(\eta)})
    .\frac{1}{\sqrt{{i_\Gamma}^*\td(\pr_{\cX_\eta}^*T_{\cX_\eta|k(\eta)}\oplus \pr_{\cX'_\eta}^*T_{\cX'_\eta|k(\eta)})}}\Big)\\
    &={i_\Gamma}_*\Big(
    \td(T_{\Gamma|k(\eta)})
    .\frac{1}{\sqrt{\td(T_{\Gamma|k(\eta)}).\td(T_{\Gamma|k(\eta)})}}\Big)\\
    &={i_\Gamma}_*(1_{\Gamma})=[\Gamma]\\
    &=\alpha_{\eta}
    .
  \end{align*}
  This concludes the proof.
\end{proof}
  Since the Chern character $\ch:K(X)\otimes \bQ
  \to \CH_\bQ(X)$ is an isomorphism of rings (see \cite[Theorem 11.6]{Manin}), one can conclude from
  \eqref{diag:FM-diagram}, Corollary \ref{cor:[Z]_*(td)=td}, Theorem \ref{thm:invertibility} and Proposition \ref{prop:[Z]=v(Oz)} that:
  \begin{corollary}\label{lem:mult_on_K-groups}
    On the level of rational Grothendieck rings
    the map $\Phi^K_{v(\dO_Z)}$ is an isomorphism and in particular multiplicative.% :
    % \begin{equation*}
    %   \Phi_{v(\dO_Z)}^K(e)\,.\,\Phi_{v(\dO_Z)}^K(f)
    %   =\Phi_{v(\dO_Z)}^K(e.f)
    %   \hspace{3 em} \forall \, e, f \in K(X)\otimes \bQ.
    % \end{equation*} 
    \hfill  \qedsymbol
  \end{corollary}

  Consider a general Mukai flop $X\leftarrow \Xh \to X'$. Combining Theorem \ref{thm:Namikawa} with Balmer's
  result in \cite{Balmer} shows that the Fourier--Mukai transform $\Phi_{\dO_Z}$ is not compatible with
  the derived tensor product. Therefore, even in this special case, there is a priori no reason on the
  level of derived categories that $\Phi^K_{[\dO_Z]}$ or $\Phi^{\CH}_{v(\dO_Z)}$ should respect the ring
  structure.

%%%%%%%%%%%%%%%%%%%%%%%%%%%%%%%%%%%%% References %%%%%%%%%%%%%%%%%%%%%%%%%%%%%%%%%%%%%
%%%%%%%%%%%%%%%%%%%%%%%%%%%%%%%%%%%%%%%%%%%%%%%%%%%%%%%%%%%%%%%%%%%%%%%%%%%%%%%%
\bibliographystyle{alpha}
\bibliography{Literatur}

\begin{thebibliography}{{Huy}03a}

\bibitem[Art70]{Artin:blow-down}
Michael Artin.
\newblock {Algebraization of formal moduli: II. - Existence of modifications.}
\newblock {\em Ann. Math. (2)}, 91:88--135, 1970.

\bibitem[Bal02]{Balmer}
Paul Balmer.
\newblock {Presheaves of triangulated categories and reconstruction of
  schemes.}
\newblock {\em Math. Ann.}, 324(3):557--580, 2002.

\bibitem[Bea83]{Beauville1983}
Arnaud Beauville.
\newblock {Vari\'et\'es k\"ahleriennes dont la premi\`ere classe de Chern est
  nulle}.
\newblock {\em J. Differ. Geom.}, 18:755--782, 1983.

\bibitem[Bea07]{Beauville}
Arnaud Beauville.
\newblock On the splitting of the {B}loch--{B}eilinson filtration.
\newblock In Jan Nagel and Chris Peters, editors, {\em Algebraic Cycles and
  Motives, (vol. 2)}, London Mathematical Society lecture note series 344,
  pages 38--53. Cambridge University Press, 2007.

\bibitem[BV04]{Beauville-Voisin}
Arnaud Beauville and Claire Voisin.
\newblock {On the {C}how ring of a {K}3 surface.}
\newblock {\em J. Algebr. Geom.}, 13(3):417--426, 2004.

\bibitem[EG98]{Edidin-Graham}
Dan {Edidin} and William {Graham}.
\newblock {Equivariant intersection theory (With an appendix by Angelo Vistoli:
  The Chow ring of ${\mathcal M}_2$).}
\newblock {\em {Invent. Math.}}, 131(3):595--644, 1998.

\bibitem[FN72]{Fuj-Nak-72}
Akira Fujiki and Shigeo Nakano.
\newblock {Supplement to "On the inverse of monoidal transformation".}
\newblock {\em Publ. Res. Inst. Math. Sci., Kyoto Univ.}, 7:637--644, 1972.

\bibitem[Ful84]{Fulton}
William Fulton.
\newblock {\em Intersection Theory}.
\newblock Springer-Verlag, second edition, 1984.

\bibitem[FW08]{FuWang08}
Baohua Fu and Chin-Lung Wang.
\newblock {Motivic and quantum invariance under stratified {M}ukai flops}.
\newblock {\em J. Diff. Geom.}, 80(2):261--280, 2008.

\bibitem[GHJ03]{Gross-Huybrechts-Joyce}
Mark Gross, Daniel Huybrechts, and Dominic Joyce.
\newblock {\em {Calabi--Yau Manifolds and Related Geometries. Lectures at a
  summer school in Nordfjordeid, Norway, June 2001.}}
\newblock Universitext. Springer, Berlin, 2003.

\bibitem[Har77]{Hartshorne}
Robin Hartshorne.
\newblock {\em Algebraic Geometry}.
\newblock Graduate Texts in Mathematics. Springer, New York, 1977.

\bibitem[Huy97]{Huybrechts:birHK_deformations}
Daniel Huybrechts.
\newblock {Birational symplectic manifolds and their deformations}.
\newblock {\em J. Differ. Geom.}, 45(3):488--513, 1997.

\bibitem[Huy99]{Huybrechts:HK:basic-results}
Daniel Huybrechts.
\newblock {Compact hyperk\"ahler manifolds: Basic results}.
\newblock {\em Invent. Math.}, 135(1):63--113, 1999.

\bibitem[{Huy}03a]{Huybrechts:HK:basic-results:erratum}
Daniel {Huybrechts}.
\newblock {Erratum to: Compact hyperk\"ahler manifolds: basic results}.
\newblock {\em {Invent. Math.}}, 152(1):209--212, 2003.

\bibitem[Huy03b]{Huybrechts2003}
Daniel Huybrechts.
\newblock {The K\"ahler cone of a compact hyperk\"ahler manifold}.
\newblock {\em Math. Ann.}, 326(3):499--513, 2003.

\bibitem[Huy06]{Huybrechts:FM-transforms}
Daniel Huybrechts.
\newblock {\em {Fourier--Mukai Transforms in Algebraic Geometry}}.
\newblock Oxford Mathematical Monographs. {Oxford Science Publications}, 2006.

\bibitem[Jan94]{Jannsen:Bloch-Beilinson-conjecture}
Uwe Jannsen.
\newblock {Motivic sheaves and filtrations on {C}how groups}.
\newblock {Jannsen, Uwe (ed.) et al., Motives. Proceedings of the summer
  research conference on motives, held at the University of Washington,
  Seattle, WA, USA, July 20-August 2, 1991. Providence, RI: American
  Mathematical Society. Proc. Symp. Pure Math. 55, Pt. 1}, 1994.

\bibitem[{Kaw}92]{Kawamata92}
Yujiro {Kawamata}.
\newblock {Unobstructed deformations -- a remark on a paper of Z. Ran}.
\newblock {\em {J. Algebr. Geom.}}, 1(2):183--190, 1992.

\bibitem[Kaw02]{Kawamata02}
Yujiro Kawamata.
\newblock {$D$-equivalence and $K$-equivalence}.
\newblock {\em J. Diff. Geom.}, 61(1):147--171, 2002.

\bibitem[Knu71]{Knutson}
Donald Knutson.
\newblock {\em Algebraic Spaces}, volume 203 of {\em Lecture Notes in
  Mathematics}.
\newblock Springer-Verlag, 1971.

\bibitem[Laz04]{LazarsfeldI}
Robert Lazarsfeld.
\newblock {\em {Positivity in Algebraic Geometry I. Classical Setting: Line
  Bundles and Linear Series}}, volume~48 of {\em Ergebnisse der Mathematik und
  ihrer Grenzgebiete. 3. Folge}.
\newblock Springer, 2004.

\bibitem[LLW10]{LeeLinWang}
Yuan-Pin Lee, Hui-Wen Lin, and Chin-Lung Wang.
\newblock {Flops, motives, and invariance of quantum rings}.
\newblock {\em Ann. Math. (2)}, 172(1):243--290, 2010.

\bibitem[Man69]{Manin}
Yuri~I. Manin.
\newblock {Lectures on the {K}-functor in algebraic geometry}.
\newblock {\em Russ. Math. Surv.}, 24(5):1--89, 1969.

\bibitem[Muk84]{Mukai}
Shigeru Mukai.
\newblock Symplectic structure of the moduli space of sheaves on an abelian or
  {K}3 surface.
\newblock {\em Invent. math.}, 77:101--116, 1984.

\bibitem[Mum69]{Mumford69}
David Mumford.
\newblock {Rational equivalence of $0$-cycles on surfaces.}
\newblock {\em J. Math. Kyoto Univ.}, 9:195--204, 1969.

\bibitem[Nak71]{Nakano-71}
Shigeo Nakano.
\newblock {On the inverse of monoidal transformation.}
\newblock {\em Publ. Res. Inst. Math. Sci., Kyoto Univ.}, 6:483--502, 1971.

\bibitem[Nam03]{Namikawa2003}
Yoshinori Namikawa.
\newblock {Mukai flops and derived categories}.
\newblock {\em J. Reine Angew. Math.}, 560:65--76, 2003.

\bibitem[{Ran}92]{Ran92}
Ziv {Ran}.
\newblock {Deformations of manifolds with torsion or negative canonical
  bundle}.
\newblock {\em {J. Algebr. Geom.}}, 1(2):279--291, 1992.

\bibitem[Voi07]{VoisinI}
Claire Voisin.
\newblock {\em {Hodge Theory and Complex Algebraic Geometry, I}}, volume~76 of
  {\em Cambridge Studies in Advanced Mathematics}.
\newblock {Cambridge University Press}, 2007.

\bibitem[Voi08]{Voisin-08}
Claire Voisin.
\newblock {On the Chow ring of certain algebraic hyper-K\"ahler manifolds}.
\newblock {\em Pure Appl. Math. Q.}, 4(3):613--649, 2008.

\bibitem[Wie02]{Wierzba}
Jan Wierzba.
\newblock Birational geometry of symplectic 4-folds.
\newblock Unpublished preprint: http://www.mimuw.edu.pl/\textasciitilde
  jarekw/postscript/bir4fd.ps, 2002.

\end{thebibliography}

\end{document}